\documentclass[12pt]{amsart}
\PassOptionsToPackage{usenames,dvipsnames}{xcolor}
\usepackage{txfonts}
\PassOptionsToPackage{hyphens}{url}\usepackage{hyperref}
\usepackage{subfig}
\usepackage{amscd,amssymb,mathabx}
\usepackage[arrow,matrix,graph,frame,poly,arc,tips]{xy}
\usepackage{graphicx,calrsfs}
\usepackage{mathtools}
\usepackage{tikz}
\usetikzlibrary{decorations.markings}
\usetikzlibrary{shapes}
\usetikzlibrary{arrows}
\usetikzlibrary{backgrounds}
\usetikzlibrary{calc}
\usepackage{mdwlist}
\usepackage{tikz-cd} 
\usepackage{multirow}
\usepackage{enumerate}
\usepackage{verbatim}
\usepackage{etoolbox}
\AtEndEnvironment{proof}{\setcounter{claim}{0}}

\makeatletter
\@namedef{subjclassname@2020}{\textup{2020} Mathematics Subject Classification}
\makeatother

\newcommand\ba{\begin{align*}}
\newcommand\ea{\end{align*}}
\newcommand\be{\begin{enumerate}}
\newcommand\ee{\end{enumerate}}
\newcommand\bp{\begin{proof}}
\newcommand\ep{\end{proof}}
\newcommand\bpp{\begin{prop}}
\newcommand\epp{\end{prop}}
\newcommand\bpb{\begin{prob}}
\newcommand\epb{\end{prob}}
\newcommand\bd{\begin{defn}}
\newcommand\ed{\end{defn}}
\newcommand\bh{\begin{hint}}
\newcommand\eh{\end{hint}}

\newcommand\mS{\mathcal{S}}

\newcommand\bN{\mathbb{N}}

\newcommand\bR{\mathbb{R}}

\newcommand\bQ{\mathbb{Q}}

\newcommand\bZ{\mathbb{Z}}

\newcommand\CC{\mathcal{C}}

\newcommand\GG{\mathcal{G}}

\newcommand\KK{\mathcal{K}}

\newcommand\Hom{\operatorname{Hom}}

\newcommand\supp{\operatorname{supp}}

\newcommand\lk{\operatorname{Lk}}
\newcommand\st{\operatorname{St}}

\newcommand\gam{\Gamma}
\newcommand\Mod{\operatorname{Mod}}

\DeclareMathOperator\Out{Out}
\DeclareMathOperator\Aut{Aut}

\def\thetitle{{Geometry and combinatorics via right-angled Artin groups}}
\def\theauthors{{Thomas Koberda}}
\usepackage{hyperref}
\hypersetup{
   colorlinks=false,
   plainpages,
   urlcolor=black,
   linkcolor=black
   pdftitle=   \thetitle,
   pdfauthor=  {\theauthors}
}

\theoremstyle{plain}
\newtheorem{thm}{Theorem}[section]
\newtheorem{lem}[thm]{Lemma}
\newtheorem{cor}[thm]{Corollary}
\newtheorem{prop}[thm]{Proposition}

\newtheorem{que}[thm]{Question}
\newtheorem*{principle*}{Principle}
\newtheorem*{claim*}{Claim}

\theoremstyle{remark}

\theoremstyle{definition}
\newtheorem{defn}[thm]{Definition}
\newtheorem{prob}{Problem}[section]

\begin{document}
\title\thetitle
\date{\today}

\author[T. Koberda]{Thomas Koberda}
\address{Department of Mathematics, University of Virginia, Charlottesville, VA 22904-4137, USA}
\email{thomas.koberda@gmail.com}
\urladdr{https://sites.google.com/view/koberdat}

\begin{abstract}
We survey the relationship between the combinatorics and geometry of graphs and the algebraic structure of right-angled Artin groups.
We concentrate on the defining graph of the right-angled Artin group and on the extension graph associated to the right-angled Artin
group. Additionally, we discuss connections to geometric group theory and complexity theory. The final version of this survey will
appear in ``In the tradition of Thurston, vol.~II", ed.~K.~Ohshika and A.~Papadopoulos.
\end{abstract}

\maketitle
\tableofcontents

\section{Introduction}
In this paper, we survey the interplay between the algebraic structure of right-angled Artin groups, the combinatorics of graphs,
and geometry. Throughout the paper, let $\gam$ be a finite simplicial graph, and we write $V(\gam)$ and $E(\gam)$ for the set
of vertices and edges of $\gam$, respectively. The \emph{right-angled Artin group}\index{right-angled Artin group}
on $\gam$, denoted by $A(\gam)$, is the group
defined by \[A(\gam)=\langle V(\gam)\mid [v,w]=1\textrm{ if and only if } \{v,w\}\in E(\gam)\rangle.\] 

\subsection{Scope of this survey}
Right-angled Artin groups interpolate between free groups and abelian groups, and they exhibit a wide range of complex phenomena.
Moreover, they are simple enough that their structure is relatively tractable, and hence one can come to understand these groups
fairly well. They are prototypical examples of CAT(0) groups, and they serve as toy examples that mirror many important properties of
and inform conjectures about more complicated groups, such as mapping class groups. Some 
well-known and difficult conjectures about mapping class groups,
such as the characterization of convex cocompact subgroups, admit complete, tractable analogues in the case of right-angled Artin
groups; see~\cite{KMT2014} for a detailed discussion.

In this article, we will concentrate on some specific aspects of right-angled Artin groups, which we will outline in the remainder of this
section. For a survey of the general properties of
right-angled Artin groups, the reader is directed to~\cite{Charney2007}.

Some of the basic questions we will discuss are as follows.

\begin{que}\label{q:comb-intro}
What is the exact relationship between the group theoretic structure of the group $A(\gam)$ and the combinatorial structure of $\gam$?
\end{que}

The reader will find that there are two answers to Question~\ref{q:comb-intro}, the trivial one and the nontrivial one. The trivial one
will be a consequence of Theorem~\ref{thm:coho-alg} below, which shows that $\gam$ is completely determined by the cohomology
algebra of $A(\gam)$, and in fact by the degree one and two parts together with the cup product pairing. Thus, one can in principle
recover $\gam$ from $A(\gam)$, so that any combinatorial properties of $\gam$ is automatically determined by the algebraic structure
of $A(\gam)$. Conversely, the algebraic structure of $A(\gam)$ is, in a sense that is
so general as to render it almost meaningless, ``known" by the graph $\gam$.

There is a more interesting approach to Question~\ref{q:comb-intro} that seeks to find a dictionary between the combinatorics of $\gam$
and the algebra of $A(\gam)$, by passing between specific graph-theoretic and group-theoretic properties that are analogous. This
line of inquiry yields some otherwise nonobvious
insights that have applications outside of geometric group theory, such as in cryptography and
complexity theory. Some sample results we will discuss in the sequel are the following:

\begin{thm}[\cite{FKK2019}]\label{thm:aut-intro}
Let $\gam$ be a finite simplicial graph. Then $\gam$ admits a nontrivial automorphism if and only if the outer automorphism group
$\Out(A(\gam))$ contains a finite nonabelian group.
\end{thm}

\begin{thm}[\cite{FKK2020a}]\label{thm:color-intro}
Let $\gam$ be a finite simplicial graph with $n$ vertices. Then $\gam$ admits a $k$-coloring if and only if $A(\gam)$ surjects to a product
\[F_{n_1}\times\cdots \times F_{n_k},\] where $F_{n_i}$ is a free group of rank $n_i$, and \[\sum_{i=1}^k n_i=n.\]
\end{thm}

\begin{thm}[\cite{FKK2021ham}]\label{thm:ham-intro}
Let $\gam$ be a finite simplicial graph. Then $\gam$ admits a Hamiltonian cycle if and only if the cohomology algebra of $A(\gam)$
is a Hamiltonian vector space.
\end{thm}

In Theorem~\ref{thm:ham-intro}, \emph{Hamiltonicity} of a vector space means that there is a bilinear form satisfying certain ``connectivity"
conditions. We direct the reader to Subsection~\ref{ss:ham} for precise definitions.

Whereas the graph $\gam$ is evidently intimately related to the structure of $A(\gam)$, the graph $\gam$ is not always ideally suited for
the study of the internal structure of $A(\gam)$, since there is no natural interesting action of $A(\gam)$ on $\gam$. However, one can
augment $\gam$ in the ``smallest way possible" in order to get a graph on which $A(\gam)$ acts. The key idea is to conflate a vertex
$v$ of $\gam$ with an element of $A(\gam)$. One can then consider the set
\[V(\gam^e)=\{v^g\mid v\in V(\gam), g\in A(\gam)\}\subset A(\gam)\]
of all conjugates of vertices of $\gam$, where here we write $v^g=g^{-1}vg$.
It is true though largely irrelevant that $V(\gam^e)$ is not canonically defined as a subset of $A(\gam)$, since automorphisms of $A(\gam)$
need not preserve the set of conjugates of given vertex generators of $A(\gam)$.

We build a graph $\gam^e$, called the \emph{extension graph}\index{extension graph} of $\gam$
(cf.~\cite{KK2013}), by putting an edge $\{v^g,w^h\}$ between vertices in
$V(\gam^e)$ whenever $[v^g,w^h]=1$ in $A(\gam)$. The group $A(\gam)$ now acts in a canonical way on $\gam^e$, i.e.~by conjugation.

\begin{que}
What is the relationship between the structure of $A(\gam)$ and the structure of $\gam^e$? What is the geometry of the action of $A(\gam)$
on $\gam^e$?
\end{que}

The graph $\gam^e$, though algebraically defined, is very closely related to Hagen's \emph{contact graph}\index{contact graph}
~\cite{Hagen14}, which encodes the intersection
pattern between hyperplanes in a natural CAT(0) cube complex on which $A(\gam)$ acts. This, together with an analogy between
the extension graph and the curve graph associated to a hyperbolic surface of finite type, is an entry point into the theory of hierarchically
hyperbolic spaces (HHSs) and hierarchically hyperbolic groups (HHGs)
(see~\cite{HHS1,HHS2}, for instance). We will largely avoid discussing that aspect of the theory in this
paper.

The extension graph carries a large amount of data about the subgroup structure of $A(\gam)$.
A sample result we will discuss is the following:

\begin{thm}[\cite{KK2013}]\label{thm:subgroup-intro}
Suppose $\gam$ has no triangles, and let $\Lambda$ be an arbitrary finite simplicial graph. Then $A(\Lambda)$ occurs as a subgroup
of $A(\gam)$ if and only if $\Lambda$ occurs as a subgraph of $\gam^e$.
\end{thm}

The action of $A(\gam)$ on $\gam^e$ by conjugation, though perhaps simple at first glance, serves to unify the group theory of $A(\gam)$,
the geometry of $\gam^e$, and the intrinsic CAT(0) geometry of $A(\gam)$. We state the following result that we will discuss in some detail,
and we will defer definitions of the terminology until then.

\begin{thm}[\cite{KK2013b}]\label{thm:acyl-intro}
The action of $A(\gam)$ on $\gam^e$ is acylindrical. Moreover, the following are equivalent.
\begin{enumerate}
\item
The element $g\in A(\gam)$ acts loxodromically on $\gam^e$.
\item
The element $g\in A(\gam)$ acts as a rank one isometry of the universal cover of the Salvetti complex of $\gam$.
\item
The element $g\in A(\gam)$ is not conjugate into a join subgroup of $A(\gam)$.
\end{enumerate}
\end{thm}

We will not give detailed proofs of most of the results in this survey. We will give proof sketches where it is feasible, and we will strive to give
complete references. As already suggested above, we will omit large parts of the theory and neglect various viewpoints. The specific topics
discussed herein undoubtedly reflect the idiosyncratic tastes of the author.

\subsection{Notation and terminology}

Most of the notation and terminology used in this survey is standard or nearly standard.
All graphs will be
undirected and simplicial unless otherwise noted, so that in particular there are no double edges nor edges that start and
end at a single vertex.
The \emph{complement}\index{complement graph}
of a graph $\gam$ is the complement of $\gam$ in the complete graph on the vertices of $\gam$; that is,
complete all the missing edges of $\gam$ and then delete the edges that were present in $\gam$. Two vertices are therefore
connected by an edge in the complement of $\gam$ if and only if they are not connected by an edge in $\gam$.

A graph $\gam$ is a 
\emph{join}\index{graph join}
if its complement graph is disconnected. The join of graphs $\gam_1$ and $\gam_2$ is written $\gam_1*\gam_2$, and every vertex of
$\gam_1$ is adjacent to every vertex of $\gam_2$. The join of two graphs mimics the geometric join in topology: if $A$ and $B$ are
topological spaces, then the join $A*B$ is the quotient of $A\times B\times I$ that collapses $A\times B\times \{0\}$ to $A$ and
$A\times B\times \{1\}$ to $B$.
For us, a subgraph $\Lambda$ of a graph $\gam$ is always \emph{full}\index{full subgraph}, which is to say
$\Lambda$ contains all edges that are present in $\gam$. A \emph{clique}\index{clique} is a complete graph, and a
\emph{$k$-clique} is a complete graph on
$k$ vertices.
The set $V(\gam)$, viewed as a subset
of $A(\gam)$, is called the set of \emph{vertex generators}\index{vertex generator}
of $A(\gam)$. The \emph{link}\index{link} of a vertex $v\in V(\gam)$ is written $\lk(v)$ and
consists of the vertices that are adjacent to $v$. If $\emptyset\neq S\subset V(\gam)$ then \[\lk(S)=\bigcap_{s\in S}\lk(s).\]
The \emph{star}\index{star} of $v$ is given by $\st(v)=\lk(v)\cup\{v\}$. The \emph{degree}\index{degree of a vertex} of a vertex $v$
is given by $|\lk(v)|$. A vertex $v$ is \emph{isolated}\index{isolated vertex} if $\lk(v)$ is empty. A graph is
\emph{totally disconnected}\index{totally disconnected graph} if every vertex is isolated.
A \emph{path}\index{path} in $\gam$ is a tuple of vertices $\overline{p}=(v_1,\ldots,v_k)$ in $V(\gam)$ such that $\{v_i,v_{i+1}\}\in E(\gam)$
for all suitable indices.
The parameter $k$ is arbitrary, and the \emph{length} of the path $\overline{p}$ is $k-1$.
A \emph{cycle} or \emph{circuit}\index{circuit} is a path for which $v_1=v_k$ and for which $v_i\neq v_{i+2}$ for all suitable indices.
A graph is \emph{connected}\index{connected graph} if for all pairs of vertices $v,w\in V(\gam)$, there
is a path in $\gam$ such that $v=v_1$ and $w=v_k$.

The \emph{rank of a linear map}\index{rank} is the dimension of its image, and the \emph{rank of a group}
is the minimal number of generators of the group. The identity element of a group is denoted $1$ with an exception in the case of
additive abelian groups when it is written $0$.

Let $\gam$ be a (possibly infinite) graph. We build a graph $\gam_k$, called the \emph{clique graph}\index{clique graph}
of $\gam$ as follows.
We start with the vertices and edges of $\gam$. For every complete subgraph $K\subset V(\gam)$ with at least two vertices,
we add a new vertex $v_K$ to
$V(\gam_k)$. If $K_1$ and $K_2$ are cliques such that $V(K_1)\cup V(K_2)$ also spans a complete subgraph of
$\gam$, then we add an edge
$\{v_{K_1},v_{K_2}\}$ to $E(\gam_k)$. Finally, we add an edge between each vertex of the form $v_K$ and the vertices making up $K$.
The resulting graph is the clique graph. It is helpful to illustrate the clique graph with an
example: if $\gam$ is a graph without triangles then the only cliques with two or more vertices are the edges of $\gam$. In this case,
the clique graph of $\gam$ is just $\gam$ with an extra vertex $v_e$ for each edge $e\in E(\gam)$, and two edges connecting
$v_e$ to the two vertices of $\gam$ spanning $e$. Thus, $\gam_k$ is just a copy of $\gam$ with a ``fin" hanging off each edge.

\subsection{A remark about generators}
When we specify a right-angled Artin group, we will write $A(\gam)$. Since $A(\gam)$ as an abstract group determines $\gam$ up to
isomorphism, the specification of $\gam$ (viewed as an abstract graph) does not constitute a choice of generators for $A(\gam)$.
However, once we speak of particular generators of $A(\gam)$, we have implicitly chosen an identification of $V(\gam)$ with
a set of generators for $A(\gam)$. The author has taken pains to avoid ambiguities that could cause confusion for the reader.

\section{The cohomology ring of a right-angled Artin group}
A central role in the dictionary between algebra and combinatorics is played by the cohomology of a right-angled Artin group.
Recall that the cohomology of a group $G$ is defined to be the cohomology of a $K(G,1)$, which is unique up to homotopy equivalence
(see~\cite{Hatcher}, for instance).
A right-angled Artin group has a very
easy to describe $K(G,1)$, and a large number of natural retractions allows for an efficient calculation of the cohomology algebra.
For the entirety of thus section, $R$ will denote a commutative ring with a unit, unless otherwise noted.

\subsection{The topology of the Salvetti complex}

We will write $\mS(\gam)$ for the Salvetti complex\index{Salvetti complex} of $\gam$, and we construct it as follows (cf.~\cite{Charney2007}).
Let $\gam$ be a graph with $n=|V(\gam)|$. We fix a bijection between $V(\gam)$ and $\{1,\ldots,n\}$.
Consider now the unit cube $[0,1]^n\subset \bR^n$. We build a certain subset
$S\subset [0,1]^n$, cube by cube. For $1\leq i\leq n$, we write $J_i$ for the unit segment in the $i^{th}$ coordinate direction, emanating
from the origin. We include $J_i$ in $S$ for all $i$. Now, if $K\subset \{1,\ldots,n\}$ consists of a collection of vertices which span a complete
subgraph of $\gam$, then we include the subcube of $[0,1]^n$ spanned by $\{J_i\}_{i\in K}$ in $S$.

Once $S$ has been constructed in this way, we set $\mS(\gam)$ to be the image of $S$ in $\bR^n/\bZ^n$, where $\bZ^n$ acts on
$\bR^n$ by usual integer translations. Thus, the complex $\mS(\gam)$ is realized as a subcomplex of an $n$-dimensional torus.

\begin{prop}\label{p:top of salv}
The following are properties of $\mS(\gam)$.
\begin{enumerate}
\item
The fundamental group of $\mS(\gam)$ is isomorphic to $A(\gam)$.
\item
The universal cover of $\mS(\gam)$ is contractible.
\end{enumerate}
\end{prop}

That the fundamental group of $\mS(\gam)$ is isomorphic to $A(\gam)$ is a straightforward calculation
using Van Kampen's Theorem. That the universal cover of $\mS(\gam)$
is contractible is much less obvious, and follows from the fact that $\mS(\gam)$ admits the structure of a locally
CAT(0) cube complex\index{CAT(0) cube complex}.
To delve into the details would take us far afield, and we shall content ourselves to direct the reader to some references, such as
~\cite{BH1999,gromov99,Wise2012}. The crucial
point here is that the homology and cohomology of $\mS(\gam)$ are in fact a invariants of $A(\gam)$, since $\mS(\gam)$ is a $K(G,1)$
for $G=A(\gam)$\index{group cohomology}.

The homology of $\mS(\gam)$ is easily calculated by a standard Mayer--Vietoris argument. In our construction of $\mS(\gam)$ above, we
obtain a distinguished $k$-subtorus of $\mS(\gam)$ for every $k$-subclique of $\gam$. When two such distinguished
subtori (corresponding to subcliques $K_1$ and $K_2$ of $\gam$) meet, they meet along the distinguished subtorus corresponding
to the intersection $K_1\cap K_2$ (which is just the basepoint in case this intersection is empty). Thus, we see that:

\begin{prop}\label{p:raag-homo}
Let $R$ be a ring. Then $H_k(A(\gam),R)\cong R^{N_k}$, where $N_k$ denotes the number of $k$-cliques in $\gam$.
\end{prop}

Here and throughout, we always assume that the $A(\gam)$ action on the ring of coefficients
is trivial, so that our homology and cohomology
groups are always untwisted. Computation of the twisted groups is much more complicated; cf.~\cite{Davis98,MJ05}.
In particular, the rank of the abelianization of $A(\gam)$ is the number of vertices of $\gam$, and the dimension of the second homology
coincides with the number of edges.

The cohomology groups of $A(\gam)$ have the same ranks as the homology groups, and the formal structure of $A(\gam)$ (or
of $\mS(\gam)$)
allows one to give a satisfactory description of the cohomology algebra of $A(\gam)$. For this, we let $T_k$ denote the $k$-dimensional torus.
As is standard, the cohomology algebra of $T_k$ with coefficients in $R$ is $\bigwedge(R^k)$, the exterior algebra of $R^k$.

\begin{prop}\label{prop:retraction}
Let $\Lambda\subset\gam$ be a subgraph. Then the map $A(\gam)\longrightarrow A(\Lambda)$ defined by the identity for vertices
$\lambda\in V(\Lambda)$ and by $v\mapsto 1$ otherwise is a retraction of groups.
\end{prop}

Of course, the fact that Salvetti complexes are classifying spaces for right-angled Artin groups means that Proposition~\ref{prop:retraction}
admits a dual statement for spaces. That is, there are natural retractions 
$\mS(\gam)\longrightarrow\mS(\Lambda)$ which induce the corresponding
maps on fundamental groups whenever $\Lambda\subset\gam$ is a subgraph.

Specializing to the case where $K\subset\gam$ is a $k$-clique, we get a natural
surjective map \[A(\gam)\longrightarrow A(K)\cong\bZ^k,\] and thus an induced
injective map on cohomology $\bigwedge(R^k)\longrightarrow H^*(A(\gam),R)$. Suppose
we have a decomposition of graphs $\gam=\Lambda_1\cup\Lambda_2$ and $\Theta=\Lambda_1\cap\Lambda_2$. For technical reasons,
we suppose that every edge between $\Lambda_1$ and $\Lambda_2$ is realized by an edge between $\Lambda_i$ and $\Theta$ for
$i\in\{1,2\}$.
We obtain a
natural commutative diagram of retractions.

\[\begin{tikzcd}
A(\gam) \arrow{r} \arrow{d} & A(\Lambda_1) \arrow{d} \\
A(\Lambda_2) \arrow{r} & A(\Theta)
\end{tikzcd}
\]

Replacing the retractions by inclusions of groups, $A(\gam)$ acquires the structure of a graph of groups with vertex groups $A(\Lambda_1)$
and $A(\Lambda_2)$ and edge group $A(\Theta)$ (cf.~\cite{Serre1977}).
Without the assumption that every edge between $\Lambda_1$ and $\Lambda_2$ be
realized by $\Theta$, this previous assertion would no longer be true.

Dualizing, we get a commutative diagram on cohomology.

\[\begin{tikzcd}
H^*(A(\gam),R) & H^*(A(\Lambda_1),R)\arrow{l} \\
H^*(A(\Lambda_2),R)\arrow{u}  & H^*(A(\Theta),R)\arrow{l}\arrow{u}
\end{tikzcd}
\]

In category theory language, $H^*(A(\gam),R)$ is the pushout of the corresponding diagram. Again, the technical hypothesis
on the decomposition of $\gam$ is hidden in this last assertion, since the assertion follows from the Mayer--Vietoris sequence and would
be false without this hypothesis (cf.~ for example when $\gam$ is a complete graph and $\Lambda_1$ and $\Lambda_2$ are both proper
subgraphs).

These considerations show that one can
describe the cohomology algebra of $A(\gam)$ entirely in terms of exterior algebras by inductively building up $\gam$ from its cliques.
In particular, one can take an appropriate exterior algebra for each maximal clique in $\gam$, and identify exterior subalgebras
corresponding to intersections of maximal cliques.
The simplest cliques are the $1$-cliques, and a retraction 
\[A(\gam)\longrightarrow\langle v\rangle\cong\bZ\] for $v\in V(\gam)$ allows us to identify
preferred generators $\{v^*\mid v\in V(\gam)\}$ for $H^1(A(\gam),R)$, which we will refer to as the \emph{dual $1$-classes} to the
vertex generators. These dual $1$-classes can be interpreted as dual to certain natural subspaces of $\mS(\gam)$, though we will
not require this point of view here.

We clearly have that $H^*(A(\gam),R)$ is generated by its degree one
part. Now let $v,w\in V(\gam)$ and let $v^*$ and $w^*$ be the corresponding dual $1$-classes. There is a retraction
$A(\gam)\longrightarrow \langle v,w\rangle$, and the target group is either $\bZ^2$ or $F_2$, corresponding to the cases where
$\{v,w\}\in E(\gam)$ and where $\{v,w\}\notin E(\gam)$ respectively. In the first case, the cup product $v^*\cup w^*$ is nontrivial
and in the second case, the cup product vanishes.

The most important consequence of the previous discussion for us in the sequel is the following,
which characterizes the degree one and degree two
parts of the cohomology of $A(\gam)$ together with the cup product pairing:

\begin{prop}\label{prop:coho}
Let $\gam$ be a finite simplicial graph with $V(\gam)=\{v_1,\ldots,v_n\}$ and let $E(\gam)=\{e_1,\ldots,e_m\}$. Then there are bases
$\{v_1^*,\ldots,v_n^*\}$ and $\{e_1^*,\ldots,e_m^*\}$ for $H^1(A(\gam),R)$ and $H^2(A(\gam),R)$ respectively, such that:
\begin{enumerate}
\item
$v_i^*\cup v_j^*=0$ if $\{v_i,v_j\}\notin E(\gam)$.
\item
$v_i^*\cup v_j^*=\pm e_{\ell}^*$ if $\{v_i,v_j\}=e_{\ell}$.
\end{enumerate}
\end{prop}

The description of $H^1(A(\gam),R)$ and $H^2(A(\gam),R)$ furnished by Proposition~\ref{prop:coho} will be essential in describing many of
the correspondences between the group theoretic structure of $A(\gam)$ and the combinatorics of $\gam$.

\subsection{Vector spaces with a vector-space valued pairing}

In the sequel, it is sometimes convenient to consider vector spaces equipped with a bilinear vector-space valued pairing.
We will write $q\colon V\times V\longrightarrow W$ for such a pairing, where 
$V$ and $W$ are both finite dimensional vector spaces over the same
field $F$. The pairing $q$ is intended to generalize the cup product pairing \[\cup\colon H^1(A(\gam),F)\times H^1(A(\gam),F)
\longrightarrow H^2(A(\gam),F),\] and so we will always adopt
the assumption that $q$ is either symmetric or anti-symmetric unless otherwise noted. This assumption on $q$ is mostly for convenience,
since relaxing some sort of symmetry assumption only adds unnecessary layers of complication that do not enrich the underlying theory
in a meaningful way.

We will say that the triple $(V,W,q)$ is \emph{pairing-connected}\index{pairing-connectedness},
if for all nontrivial direct sum decompositions $V\cong V_0\oplus V_1$,
there are vectors $v_0\in V_0$ and $v_1\in V_1$ such that $q(v_0,v_1)\neq 0$. With this terminology, we can formulate and prove
an entry in the algebra-combinatorics dictionary.

\begin{prop}[See~\cite{FKK2020exp}]\label{prop:pairing-conn}
Let $\gam$ be a finite simplicial graph, let \[V=H^1(A(\gam),F),\quad W=H^2(A(\gam),F),\] and let $q$ be the cup product pairing. Then $\gam$
is connected if and only if $(V,W,q)$ is pairing-connected.
\end{prop}

Connectedness of $\gam$ has another, simpler characterization in terms of $A(\gam)$, as we shall
indicated below; namely, $\gam$ is connected if and only if $A(\gam)$ is freely indecomposable; see Theorem~\ref{thm:free-prod}.
The (mostly complete) proof of Proposition~\ref{prop:pairing-conn} will illustrate the principle that many results that related the algebra of
$A(\gam)$ with the combinatorics of $\gam$ have an easy direction and a less easy direction.

\begin{proof}[Proof of Proposition~\ref{prop:pairing-conn}]
Suppose first that $(V,W,q)$ is pairing-connected, and let $\gam=\gam_0\cup \gam_1$ be a purported separation of $\gam$. Let
$V_i$ denote the span of the vertices $\{v_j^*\mid v_j\in V(\gam_i)\}$ for $i\in\{0,1\}$. Pairing connectedness implies that there are
vectors $w_i\in V_i$ such that $q(w_0,w_1)\neq 0$. Writing $w_0$ and $w_1$ in terms of the preferred basis vectors, we see that
there are vertices $x_i\in V(\gam_i)$ such that $q(x_0^*,x_1^*)\neq 0$, which implies that $\{x_0,x_1\}\in E(\gam)$ by Proposition
~\ref{prop:coho}, a contradiction.

Suppose conversely that $\gam$ is connected, and let $V\cong V_0\oplus V_1$ be a nontrivial direct sum decomposition that witnesses
the failure of $(V,W,q)$ to be pairing-connected. Let $\{x_1,\ldots,x_m\}$ be a sequence of vertices of $\gam$ such that every vertex
of $\gam$ appears on this list, and such that for all suitable $i$ we have $\{x_i,x_{i+1}\}\in E(\gam)$. We allow this list to have repeats.

Let \[w_0=\sum_{i=1}^n \alpha_i v_i^*\in V_0,\, w_1=\sum_{i=1}^n\beta_i v_i^*\in V_1\] be expressions for nonzero vectors with respect
to the standard dual basis for $V$. If $\{v_i,v_j\}\in E(\gam)$ then the expression $q(w_0,w_1)=0$ implies that $\alpha_i\beta_j=\alpha_j
\beta_i$. The two sides of this last equation are either both zero or both nonzero, and in the latter case we have that the pairs
$(\alpha_i,\alpha_j)$ and $(\beta_i,\beta_j)$ are proportional.
In this case, since $\{v_1^*,\ldots,v_n^*\}$ is a basis for $V$, we may perturb
$w_0$ or $w_1$ within the respective vector spaces $V_0$ and $V_1$ in order to obtain vectors for which the coefficients corresponding
to $v_i^*$ and $v_j^*$ are not proportional. Thus, the condition $q(w_0,w_1)$ implies that \[\alpha_i\beta_j=\alpha_j
\beta_i=0.\]

With these observations, we can complete the proof. Let $w_0$ be as above. Relabeling if necessary, we have $v_1=x_1$ and $v_2=x_2$.
Without loss of generality, we may assume that $\alpha_1\neq 0$. Now let $w_1\in V_1$ be expressed as above. If $\beta_2\neq 0$ then
$\alpha_1\beta_2\neq 0$, a conclusion that was ruled out by the considerations in the previous paragraph. Thus, $\beta_2=0$, and since
$w_1$ was arbitrary, the coefficient of $x_2^*$ vanishes for all vectors in $V_1$. Then, we may find a vector in $V_0$ whose coefficient
$\alpha_2$ is nonzero, and arguing symmetrically, we see that the coefficient $\beta_1$ is zero for all vectors in $V_1$. By induction on
$m$ and using the fact every vertex of $\gam$ occurs on the list $\{x_1,\ldots,x_m\}$, we see that $V_1$ must be the zero vector space.
This is a contradiction.
\end{proof}

\subsection{The cohomology ring of $A(\gam)$ determines $\gam$}

We are now ready to state and improve a central fact about the cohomology of $A(\gam)$, namely that it determines the isomorphism
type of $\gam$.

\begin{thm}\label{thm:coho-alg}
Let $\gam$ be a finite simplicial graph, let $V=H^1(A(\gam),F)$, let $W=H^2(A(\gam),F)$, and let $q$ be the cup product pairing. Then the
triple $(V,W,q)$ determines $\gam$ up to isomorphism.
\end{thm}

One essential point in Theorem~\ref{thm:coho-alg} is that the triple $(V,W,q)$ is considered abstractly, without any further data such as bases.
Before giving a proof of Theorem~\ref{thm:coho-alg}, we can make several observations about special instances of the result. First,
the dimension of $V=H^1(A(\gam),F)$ coincides with $|V(\gam)|$, and the dimension of $W=H^2(A(\gam),F)$ coincides with $|E(\gam)|$, as is
immediate from Proposition~\ref{prop:coho}. Moreover, the first and second cohomology of $A(\gam)$ together with the cup product pairing
identify complete graphs. To see this, it is convenient to introduce a map $V\longrightarrow\Hom(V,W)$,
defined by $v\mapsto f_v$, and where $f_v(v')=v\cup v'$. The graph $\gam$ is complete if and only if for all $v\in H^1(A(\gam),F)$, the
rank of the image of $f_v$ is $\dim H^1(A(\gam),F)-1$. We leave the verification of this last claim as a straightforward exercise for the
reader.

The fact that $A(\gam)$ determines the graph $\gam$ uniquely is well-known. See~\cite{Droms87,Sabalka09,Koberda2012} for various
perspectives. The proof offered here that gives uniqueness of $\gam$ via the cohomology algebra of $A(\gam)$ fits into the theory
of \emph{cohomological uniqueness}\index{cohomological uniqueness}. In the context of cohomological uniqueness, one is often
concerned with the question of whether or not a particular space (often decorated with adjectives such as $p$--completeness, where
$p$ is a prime) is determined up to homotopy equivalence by its cohomology (with various groups of coefficients). In our setting,
Theorem~\ref{thm:coho-alg} implies that among Salvetti complexes associated to finite simplicial graphs,
the integral (or rational) cohomology of
the space determines the space up to homotopy equivalence, and its defining graph up to isomorphism. Moreover, only
the ring structure on the cohomology algebra is required, and only in degrees one and two. The reader is directed
to~\cite{viruel98,viruel01,drv13,rv13} for a more detailed discussion of cohomological uniqueness.

\begin{proof}[Proof of Theorem~\ref{thm:coho-alg}]
We will actually prove a stronger statement. Suppose $\gam\longrightarrow\Lambda$ is obtained by deleting vertices
(so that $\Lambda$ is a subgraph of $\gam$), with an induced retraction
$A(\gam)\longrightarrow A(\Lambda)$ defined by sending the vertices $V(\gam)\setminus V(\Lambda)$ to the identity.
Thus, we obtain triples $(V_{\gam},W_{\gam},q_{\gam})$ and 
$(V_{\Lambda},W_{\Lambda},q_{\Lambda})$
corresponding to the cohomologies of these groups, and a map of triples
\[i_{\Lambda,\gam}\colon (V_{\Lambda},W_{\Lambda},q_{\Lambda})\longrightarrow (V_{\gam},W_{\gam},q_{\gam}),\]
which is injective on the level of vector spaces, and $q_{\Lambda}$ is extended by $q_{\gam}$.

{\bf Claim:} The triple
\[\{(V_{\Lambda},W_{\Lambda},q_{\Lambda}), (V_{\gam},W_{\gam},q_{\gam}),i_{\Lambda,\gam}\}\] uniquely determines graphs
$\Lambda$ and $\gam$, together with an injection of graphs $\Lambda\longrightarrow\gam$.
The theorem will then follow from the special case where
$V(\Lambda)=\varnothing$.

We proceed by induction on $(|V(\gam)|,|V(\Lambda)|)$, ordered lexicographically, the cases where
$|V(\gam)|\in\{1,2\}$ being easy consequences of the remarks preceding the proof.
We now suppose the claim has been established for all graphs with at most $n$ vertices, and we suppose that $\gam$ has $n+1$ vertices.
We consider the (possibly trivial) subspace \[V_0\subset V=H^1(A(\gam),F)\] spanned by vectors for which $f_v$ has rank zero. It is
immediate from Proposition~\ref{prop:coho} that a vector $w\in V_0$ is in the span on vectors dual to vertices of degree zero in $\gam$.
The quotient $V/V_0$ is isomorphic to $H^1(A(\gam'),F)$, where $\gam'$ is the result of deleting all the vertices of $\gam$ that have
degree zero.

The natural map $A(\gam)\longrightarrow A(\gam')$ given by sending isolated vertices to the identity
induces a map $H^1(A(\gam',F)\longrightarrow H^1(A(\gam),F)$, which identifies
$V/V_0$ with a subspace of $V$. The cup product 
on $H^1(A(\gam),F)$ restricts to \[q\colon H^1(A(\gam',F)\times H^1(A(\gam',F)\longrightarrow W.\]
Thus, if $V_0\neq 0$ then $\gam'$ satisfies the conclusion of the claim by induction, and
$\gam$ is obtained from $\gam'$ by adding $\dim V_0$ many isolated vertices. We may therefore assume that $\gam$ has no
isolated vertices.

We now consider a vector $v\in V$ such that the rank of $f_v$ is minimized.

{\bf Case 1:}
Suppose first that the linear span $U$ of $v$
coincides with the span of a vector dual to a vertex $x$ of $\gam$, as furnished by Proposition~\ref{prop:coho}. Then $V/U$ coincides
with the first cohomology of $A(\gam')$, where $\gam'$ is obtained from $\gam$ by deleting $x$. The map $V\longrightarrow V/U$ is
induced by the inclusion $\gam'\longrightarrow \gam$ and the corresponding injection $A(\gam')\longrightarrow A(\gam)$.

Writing $Z$ for the image of
$f_v$, we have that $W/Z$ coincides with the second cohomology of $A(\gam')$, and the cup product pairing descends to a
bilinear map \[\overline{q}\colon V/U\times V/U\longrightarrow W/Z,\] which coincides with the cup product
pairing on the cohomology of $A(\gam')$. By induction,
the triple $(V/U,W/Z,\overline{q})$ determines $\gam'$ uniquely.

Let $N\subset V$ be the kernel of $f_v$. Then $N$ is spanned by the dual vector $v$ associated to the vertex $x$ and
 the duals of the vertices which are not adjacent of $x$.
If $N=U$ then $\gam$ is the join of $x$ and $\gam'$. If not, then we pass to the quotient $V/N$, which coincides with $H^1(A(\lk(x)),F)$.
Again, the map $V\longrightarrow V/N$ is induced by the inclusion of $A(\lk(x))\longrightarrow A(\gam)$.
Passing to a suitable quotient $W/Y$ of $W$ as above, we can recover the cup product pairing on the cohomology of
$A(\lk(x))$, and thus recover
$\lk(x)$, by induction.
Finally, we use the full strength of the induction hypothesis to obtain an injection $i_x\colon \lk(x)\longrightarrow\gam'$.
The graph $\gam$ is now reconstructed by attaching $x$ to each vertex in the image of $i_x$.

To complete the induction, let
\[\{(V_{\Lambda},W_{\Lambda},q_{\Lambda}), (V_{\gam},W_{\gam},q_{\gam}),i_{\Lambda,\gam}\}\] be a triple satisfying the hypotheses
of the claim. We quotient out the degree one part of the cohomology $V_{\Lambda}$ and $V_{\gam}$ by $U$, and the map
$i_{\Lambda,\gam}$ descends to the quotients by hypothesis. By induction, we obtain an injection of graphs $\Lambda'\longrightarrow\gam'$,
where the primed graphs are obtained by deleting the vertex $x$. The links of $x$ in $\gam'$ and $\Lambda'$ can be determined
as above, whence we can reconstruct $\gam$.

{\bf Case 2:}
Suppose that $v\in V$ is arbitrary such that the rank $k$ of $f_v$ is minimized, and suppose that $v$ is supported on
the duals of two or more vertices, so that \[v=\sum_{i=1}^m \alpha_ix_i^*,\] where all indices have nonzero coefficients and $m\geq 2$.
It is clear that for all $i$, the degree of $x_i$ must coincide with $k$, by an easy application of
Proposition~\ref{prop:coho}. Consider the vertices $x_1$ and $x_2$. Observe that Proposition~\ref{prop:coho}
again implies that there cannot be a vertex that is distinct from both $x_1$ and $x_2$ and
that is adjacent to $x_1$ but not to $x_2$. Thus, every vertex that is adjacent to $x_1$
and distinct from $x_2$ is
also adjacent to $x_2$. By symmetry, the same statement holds after switching the roles of $x_1$ and $x_2$.
The argument now bifurcates into two subcases, according to whether $x_1$ and $x_2$ are adjacent or not.

{\bf Subcase 1:} Suppose first that $x_1$ and $x_2$ are adjacent, and suppose that $m\geq 3$. Suppose that
$x_3$ is not adjacent to $x_1$. Then since the degrees of $x_2$ and $x_3$ are the same and coincide with the rank $k$ of
$f_v$, we have that \[|\lk(x_2)\cup\lk(x_3)|\geq k+1.\] This violates the minimality of the choice of $v$,
since then Proposition~\ref{prop:coho} implies
that the rank of $f_v$ is at least $k+1$. It follows that $x_1$ and $x_3$ are adjacent, and by symmetry we have that $x_2$ and $x_3$
are adjacent. By a straightforward induction, we have that $\{x_1,\ldots,x_m\}$ form a clique, and
for all $i$, a vertex $y\notin \{x_1,\ldots,x_m\}$ adjacent to
$x_i$ is adjacent to all the vertices $\{x_1,\ldots,x_m\}$. Observe that if \[v'=\sum_{i=1}^m \beta_ix_i^*\] is another linear combination
of dual vectors, then nonzero linear combinations $w$ of $v$ and $v'$ also satisfy that the rank of $f_w$ is equal to $k$. We set
$V_{\min}$ to be a maximal vector subspace of $V$ that contains $v$ and such that for all $0\neq w\in V_{\min}$, the rank of $f_w$
is equal to $k$.

It is straightforward now to show that $V_{\min}$ is generated by $\{x_1^*,\ldots,x_{\ell}^*\}$, where
$\{x_1,\ldots,x_{\ell}\}$ form an $\ell$-clique such that \[\lk(x_i)\setminus\{x_j\}=\lk(x_j)\setminus\{x_i\}\] for all $i$ and $j$.

We may now proceed as in Case 1 above, treating $\{x_1,\ldots,x_{\ell}\}$ as a single vertex, and replacing the subspace $U$ by the
subspace $V_{\min}$.

{\bf Subcase 2:} We now have that $x_1$ and $x_2$ are not adjacent. If $m\geq 3$, then the argument in Subcase 1 above implies
that $x_3$ is adjacent to neither $x_1$ nor $x_2$. We thus conclude that $\{x_1,\ldots,x_m\}$ form a totally disconnected subgraph
of $\gam$, and $\lk(x_i)=\lk(x_j)$ for all $i$ and $j$. We construct a vector space $V_{\min}$ as in Subcase 1 and conclude that
it is generated by $\{x_1^*,\ldots,x_{\ell}^*\}$, where $\{x_1,\ldots,x_{\ell}\}$ form a totally disconnected graph and such that the
links of any two vertices on this list coincide. We again reduce to Case 1.
\end{proof}

Some remarks about Theorem~\ref{thm:coho-alg} are in order. For one, one need only consider the degree one and degree two parts of the
cohomology and not the full cohomology algebra, and this is not surprising since a graph is determined by its vertices and its edges,
and a graph determines the corresponding right-angled Artin group. Second, in Case 2 of the proof, the vertices
$\{x_1,\ldots,x_{\ell}\}$ are indistinguishable from each other, in the sense of graph automorphisms. That is, every permutation of
$\{x_1,\ldots,x_{\ell}\}$ is realized by a graph automorphism of $\gam$, and therefore it is reasonable that one can treat this
collection of vertices as a single vertex. Moreover, in the two subcases, $\{x_1,\ldots,x_{\ell}\}$ generates either an abelian or a free
subgroup of $A(\gam)$. The full group of automorphisms of $\bZ^{\ell}$ or of $F_{\ell}$ embeds in the group $\Aut(A(\gam))$ (cf.~
Subsection~\ref{ss:aut}
below).
Finally, in the proofs of Subcases 1 and 2, we obtain a vector space $V_{\min}$, which either comes from a clique or a totally disconnected
subgraph. These two cases can be checked linear algebraically by whether the cup product pairing is trivial or not on $V_{\min}$.

\section{Translating between group theory and combinatorics}\label{sec:comb}

In this section, we will describe some of the results and ideas that go into translation between the algebraic structure of $A(\gam)$ and
the combinatorics of $\gam$. As we have remarked already, the abstract structure of $A(\gam)$ determines completely the
nature of $\gam$, passing perhaps through cohomology (Theorem~\ref{thm:coho-alg}). We will seek clean, definitive results characterizing
aspects of the combinatorial structure of $\gam$ in terms of the algebra of $A(\gam)$. In the process,
we will gain insight into both structures.

\subsection{Elementary properties}

We begin with some of the first properties of graphs, and how these properties are reflected in $A(\gam)$. In 
Proposition~\ref{prop:pairing-conn}, we have that pairing-connectedness of the triple $(V,W,q)$ characterizes the connectedness of
$\gam$. One can characterize the connectedness of $\gam$ and its complement directly from the group theory of $A(\gam)$, without
reference to the cohomology algebra, as follows.

\begin{thm}[\cite{Servatius1989}]\label{thm:join}
The group $A(\gam)$ splits as a nontrivial direct product if and only if $\gam$ splits as a nontrivial join.
\end{thm}

Recall that a graph $\gam$ splits as a nontrivial join if and only if the complement of $\gam$ is disconnected. Dually, we have the following
fact:

\begin{thm}[\cite{BradMei01}]\label{thm:free-prod}
The group $A(\gam)$ splits as a nontrivial free product if and only if $\gam$ is disconnected.
\end{thm}

Both Theorem~\ref{thm:join} and Theorem~\ref{thm:free-prod} are easy in one direction. If $\gam$ is disconnected, then
$A(\gam)$ admits a presentation of the form \[A(\gam)=\langle V(\gam_1)\cup V(\gam_2)\mid \mathrm{R}_1\cup \mathrm{R}_2\rangle,\]
where $\gam_1$ and $\gam_2$ are nonempty and disjoint subgraphs of $\gam$, and where $\mathrm{R}_i$ only contains generators
from $\gam_i$ for $i\in\{1,2\}$. It follows then immediately that $A(\gam)\cong A(\gam_1)* A(\gam_2)$.

If $\gam$ splits as a join $\gam_1*\gam_2$, then every vertex of $\gam_1$ is adjacent to every vertex of $\gam_2$. We have that
$A(\gam_1)$ and $A(\gam_2)$ are subgroups of $A(\gam)$, and together generate the whole group. Moreover, they normalize each
other and have trivial intersection (this last point is not completely trivial and requires some argument if one wishes to be pedantic, but
we shall sweep it under the rug). It follows that $A(\gam_1)$ and $A(\gam_2)$ generate a direct product.

The converse directions are more complicated, and we outline the main ideas for the convenience of the reader.

\begin{proof}[Sketch of proof of Theorem~\ref{thm:free-prod}]
We use the characterization of free products that follows from the work of Stallings~\cite{Stallings-ends,Stallings-book}.
Let $G$ be a finitely generated group with Cayley graph $X$. Recall that the set of \emph{ends}\index{end of a group}
of $G$ is the inverse limit of
$\pi_0(X\setminus K)$, where $K$ ranges over all compact subgraphs of $X$. A group has zero, one, two, or infinitely many ends.
As right-angled Artin groups are torsion-free (as follows from Proposition~\ref{p:top of salv} for instance), we have that a right-angled Artin
group $A(\gam)$ splits as a nontrivial free product if and only if it has infinitely many ends. It thus suffices to argue that a connected
graph $\gam$ yields a group with finitely many ends. For a graph with a single vertex, we have $A(\gam)$ is $\bZ$ and hence has two ends.
A straightforward argument shows that if $G$ and $H$ are both infinite groups then $G\times H$ has one end. Thus, we have that
all nontrivial joins of graphs yield right-angled Artin groups with one end, and by induction we suppose that all
connected graphs with at most
$n$ vertices yield groups with at most two ends. Let $v\in V(\gam)$. Then there is a proper subgraph $\Lambda$ of $\gam$ such that
\[\gam=\Lambda\cup_{\lk(v)}\st(v).\] We have that $\lk(v)$ is not empty since $\gam$ is connected. Thus, we have that
\[A(\gam)=A(\Lambda)*_{A(\lk(v))}A(\st(v)).\] If $\Lambda$ is connected then $A(\gam)$ is an amalgamated product of two finite-ended
groups over an infinite subgroup
(cf.~\cite{Serre1977}), whence one can prove directly that $A(\gam)$ is one-ended. If $\Lambda$ is disconnected, then
one can argue component-by-component of $\Lambda$ to obtain the same conclusion.
\end{proof}

For Theorem~\ref{thm:join}, we require a basic result about the structure of centralizers of elements in $A(\gam)$. Let $w$ be a word
in the vertices of $\gam$ and their inverses. We say that $w$ is \emph{reduced}\index{reduced word}
if $w$ cannot be shortened by applications of free
reductions and moves of the form $[v_1^{\pm1},v_2^{\pm1}]$ for $\{v_1,v_2\}\in E(\gam)$. We say that $w$ is
\emph{cyclically reduced}\index{cyclically reduced word}
if it remains reduced after allowing cyclic permutations of the letters occurring in $w$. It is true but not trivial that the moves of free reduction
and commutation solve the word problem in right-angled Artin groups, and that cyclic reduction solves the conjugacy problem (see
especially~\cite{CGW09},
cf.~\cite{cartier-foata,hm1995,vanwyk94,wrathall88}).

The \emph{support}\index{support} of $w$ is
written $\supp(w)$ and is defined to be the set of vertices which
are required (possibly inverted) to express $w$. It is not completely trivial but true that the support of $w$ is well-defined in the sense that
for reduced words, $w_1=w_2$ in $A(\gam)$ implies that $\supp(w_1)=\supp(w_2)$.

\begin{thm}[\cite{Servatius1989}]\label{thm:cent}
Let $1\neq w\in A(\gam)$ be cyclically reduced. Then the centralizer of $w$ lies in $\langle \supp(w)\cup\lk(\supp(w))\rangle$. If the
centralizer of $w$ is not cyclic then either $\lk(\supp(w))$ is nonempty, or $\supp(w)$ decomposes as a nontrivial join.
\end{thm}

Armed with Theorem~\ref{thm:cent}, we can illustrate the other direction of Theorem~\ref{thm:join}.

\begin{proof}[Proof of Theorem~\ref{thm:join}]
Suppose that $A(\gam)\cong G\times H$ for nontrivial groups $G$ and $H$. Then since $A(\gam)$ is torsion-free, we have that
every nontrivial element of $A(\gam)$ contains a copy of $\bZ^2$ in its centralizer. Writing $V(\gam)=\{v_1,\ldots,v_n\}$, we have
that $w=v_1\cdots v_n$ is cyclically reduced and has noncyclic centralizer. Moreover, $\lk(\supp(w))=\emptyset$, so that
Theorem~\ref{thm:cent} implies that $\supp(w)=\gam$ splits as a nontrivial join.
\end{proof}

Theorem~\ref{thm:cent} has several other important consequences that relate the algebraic structure of $A(\gam)$ to the combinatorics
of $\gam$. First, we have the following.

\begin{thm}\label{thm:coho-dim}
The cohomological dimension\index{cohomological dimension} of $A(\gam)$ coincides with the size of the maximal clique in $\gam$.
\end{thm}

Theorem~\ref{thm:coho-dim} follows from standard ideas about cohomological dimension (cf.~\cite{brown-book82}), using
 the description of the Salvetti complex as a union of tori together with the fact that it is
aspherical by Proposition~\ref{p:top of salv}. We have that the maximal dimensional cells in $\mS(\gam)$ have the same dimension
as the maximal size of a clique in $\gam$, say $k$. Moreover, this $k$-cell is the top dimensional cell in a subtorus of dimension $k$,
which has nontrivial cohomology in degree $k$. Finally, the retraction 
$\mS(\gam)\longrightarrow (S^1)^k$ implies that the degree $k$ cohomology
of $\mS(\gam)$ is also nontrivial. It follows that $k$ is also the cohomological dimension of $A(\gam)$.

The cohomological dimension and maximal clique size also describe the rank of a maximal abelian subgroup.

\begin{thm}\label{thm:max-ab}
The maximal clique size of $\gam$ coincides with the rank of a maximal abelian subgroup of $A(\gam)$.
\end{thm}

Theorem~\ref{thm:max-ab} is also
a consequence of general facts about cohomological dimension. Clearly, if the maximal clique size of $\gam$
is $k$ then $A(\gam)$ contains a copy of $\bZ^k$. Since $\mS(\gam)$ is $k$-dimensional and aspherical, it follows that no cover of
$\mS(\gam)$ can have fundamental group $\bZ^{k+1}$, so there are no abelian subgroups of rank exceeding $k$.

For another perspective, suppose $\gam$ is connected and $G<A(\gam)$ is an abelian subgroup of rank
$k\geq 2$. Conjugating if necessary, at least one nontrivial element
of $G$ is cyclically reduced, so that Theorem~\ref{thm:cent} implies that all nontrivial elements of $G$ are supported on a subgraph
$J$ of $\gam$ that splits as a nontrivial join. Writing $A(J)\cong A(J_1)\times A(J_2)$, we may restrict the projections 
$A(J)\longrightarrow A(J_i)$
for each $i$ to $G$.

Now, suppose first that $\gam$ has no triangles (i.e.~$3$-cliques). Then $J_1$ and $J_2$
cannot have any edges, since otherwise $\gam$ would have a triangle. It follows then that $A(J_i)$ is free for $i\in\{1,2\}$, and so
the image of $G$ in $A(J_i)$ is cyclic for each $i$. It follows that $G$ has rank at most two.
Thus, we may assume by induction that if the maximal clique
size of $\gam$ is at most $k\geq 2$ then the maximal abelian subgroup has rank
at most $k$. Supposing $\gam$ has maximal clique size $k+1$,
then $J_1$ and $J_2$ have maximal clique sizes $k_1>0$ and $k_2>0$, which satisfy $k_1+k_2\leq k+1$.
It follows by induction that the ranks
of the images of $G$ in $A(J_1)$ and $A(J_2)$ are at most $k_1$ and $k_2$, so that $G$ has rank at most $k+1$. This gives an alternate
proof of Theorem~\ref{thm:max-ab}.

The final elementary combinatorial property of graphs we will discuss is the maximal degree of a vertex. This property is essential in
the theory of expander graphs, which will be discussed below. In the sequel we will use a different characterization of the maximal
degree that is understood through cohomology, though the following is a significantly cleaner statement.

\begin{prop}\label{prop:degree}
Let $\gam$ be a graph and let $d$ denote the maximum valence of a vertex of $\gam$. Then the rank of the centralizer of a nontrivial
element of $A(\gam)$ is at most $d+1$. Conversely, if for all elements $1\neq g\in A(\gam)$ the centralizer of $g$ has rank at most $d+1$,
then the maximum degree of a vertex of $\gam$ is at most $d$.
\end{prop}

The proof of Proposition~\ref{prop:degree} is a fairly straightforward application of Theorem~\ref{thm:cent}, and we leave it as an exercise
for the reader.

\subsection{$k$-colorability}

From the point of computational complexity, one of the most basic and difficult questions
one can pose about a graph $\gam$ is about its colorability. A \emph{(vertex)
coloring}\index{vertex coloring} of a graph
$\gam$ is a function $\kappa\colon V(\gam)\longrightarrow X$, where 
$X$ is a finite set of \emph{colors}, and where $\{v,w\}\in E(\gam)$ implies that
$\kappa(v)\neq\kappa(w)$. A classical result of Brooks~\cite{diestel-book}
says that the minimal size of $X$ is at most the maximal degree of a vertex of
$\gam$ plus one. If $\gam$ is not an odd length cycle or a clique then the bound can be improved to the maximal degree of a vertex.
The minimal size of $X$ is called the~\emph{chromatic number}\index{chromatic number}
of $\gam$, and we say that $\gam$ is $|X|$-colorable.

A graph that is $2$-colorable is called ~\emph{bipartite}\index{bipartite graph}.
Determining if a graph is bipartite is easy from a computational point of view,
and can be accomplished by a sorting algorithm that runs in a period of time that is bounded by a polynomial in the size of the set of vertices.
However, the problem of determining if a graph is $3$-colorable
is extremely difficult from a computational standpoint, and is \emph{NP-complete} (see Subsection~\ref{ss:applications} below).

We remark that there is a related notion of \emph{edge coloring}\index{edge coloring},
which is a function $\epsilon\colon E(\gam)\longrightarrow X$ such that if
$v\in V(\gam)$ is incident to both $e_1$ and $e_2$, then $\epsilon(e_1)\neq \epsilon(e_2)$. It is clear that the minimal size of $X$
for a valid edge coloring is bounded
below by the maximal degree of a vertex of $\gam$. A result of Vizing~\cite{diestel-book}
 shows that $\gam$ admits an edge coloring with $|X|$ the maximal
degree of a vertex of $\gam$ plus one. Thus, giving sharp or almost sharp estimates on edge colorability of a graph is an essentially
local problem, whereas determining vertex colorability is an essentially global problem.

Let $\gam$ be a $k$-colorable graph. Choose a $k$-coloring of $\gam$, and add an edge to $\gam$ for every pair of vertices with
different colors, naming the result $\Lambda$. Observe that the vertices of $\Lambda$ are partitioned as
\[V(\Lambda)=V_1\cup\cdots\cup V_k,\] where there are no edges between vertices in $V_i$ for each $i$, and where for $i\neq j$,
each vertex of $V_i$ is adjacent to each vertex of $V_j$. It follows that $A(\Lambda)$ is a product of free groups, and that
$A(\Lambda)$ is a quotient of $A(\gam)$. It turns out that these elementary considerations characterize $k$-colorable graphs.

\begin{thm}\label{thm:color-body}
Let $\gam$ be a finite graph with $N$ vertices. Then $\gam$ is $k$-colorable if and only if there is a surjective map
\[A(\gam)\longrightarrow \prod_{i=1}^k F_{n_i},\] where for each $i$ the group $F_{n_i}$ is free of rank $n_i$, and where
\[\sum_{i=1}^k n_i=N.\]
\end{thm}

We have already established the ``only if" direction, which is easy. The reverse direction is somewhat more substantial, owing to the
fact that the surjective homomorphism need not send vertex generators of $\gam$ to a free factor of one of the free groups occurring on
the right hand side.

\begin{proof}[Sketch of proof of Theorem~\ref{thm:color-body}]
We identify the product of free groups with $A(\Lambda)$, where $\Lambda$ is a $k$-fold join of completely disconnected graphs,
say $\{\Lambda_1,\ldots,\Lambda_k\}$. If $g\in A(\Lambda)$, then $g$ can be written uniquely as a product of $g_1\cdots g_k$,
where $g_i\in A(\Lambda_i)$. One then shows that if $g=g_1\cdots g_k$ and $h=h_1\cdots h_k$ are elements of $A(\Lambda)$
that commute, then for each $i$, the images of $g_i$ and $h_i$ in the abelianization of $A(\Lambda_i)$ are rational multiples of
each other.

Now, the surjective map $A(\gam)\longrightarrow A(\Lambda)$ induces an isomorphism 
\[\phi\colon H_1(A(\gam),\bQ)\longrightarrow H_1(A(\Lambda),\bQ),\]
which can be expressed as a matrix $A$ with respect to the vertex generators of both graphs. We will view the rows of $A$
as expressions for $\phi(v)$ for $v\in V(\gam)$, in terms of the vertex generators of $\Lambda$. We arrange the columns so
that the first $|V(\Lambda_1)|$ columns correspond to vertices of $\Lambda_1$, followed by the vertices of $\Lambda_2$, and so on.

Write $A=(A_1\mid\cdots\mid A_k)$, where the columns of $A_i$ correspond to the vertices of $\Lambda_i$, and therefore
the column space of $A_i$ has dimension $|V(\Lambda_i)|$. Note that since $A$ is invertible, the row space of $A_1$ has
dimension $|V(\Lambda_1)|=n_1$.

It is an exercise in linear algebra to show that there is an $n_1\times n_1$ minor
$B_1$ of $A_1$ and a $(N-n_1)\times (N-n_1)$ minor $C$
of $(A_2\mid\cdots\mid A_k)$ such that both $B_1$ and $C$ are invertible.

By induction, one permutes the rows of $A$ to obtain a block matrix $B=(B_1\mid\cdots \mid B_k)$ such that the diagonal
$n_i\times n_i$ blocks
$\{C_1,\ldots,C_k\}$ of $B$ are invertible. This row permutation is simply a permutation of the vertices of $\gam$.
One defines a coloring of the vertices by setting $\kappa(v_i)=j$ if in the matrix expression $B$ of $\phi$, we have that
the row $\phi(v_i)$ meets the block $C_j$. That is, the vertices corresponding to the first $n_1$ rows are assigned color $1$, the next
$n_2$ are assigned color $2$, and so on.

To check that this is a valid coloring, suppose $v$ and $w$ are adjacent in $\gam$. Then $[v,w]=1$ in $A(\gam)$. For each
block $B_i$, we may consider the restriction of the rows $\phi(v)$ and $\phi(w)$ to the columns columns in $B_i$. In $B_i$,
these two rows are rational multiples of each other.
If $v$ and $w$ were assigned the same color then in some block $B_i$, the rows both meet $B_i$ in the diagonal sub-block $C_i$. Since
$C_i$ is invertible, this is a contradiction. Thus, we see that adjacent vertices of $\gam$ are assigned different colors, and so the coloring
of $\gam$ is valid.
\end{proof}

Unpacking the final check that $\kappa$ is a valid coloring in the proof of Theorem~\ref{thm:color-body}, it is not difficult to see that
in fact one can relax the condition that the homomorphism $A(\gam)\longrightarrow A(\Lambda)$ be surjective, and replace it with the condition
that it be surjective on the level of rational homology. From a practical point of view, this is a useful observation. Indeed, checking
that a linear map is surjective is relatively easy, but maps to direct products of free groups are much less well-behaved, since the subgroup
structure of the latter is very complicated~\cite{miha1958}.

\subsection{Hamiltonicity}\label{ss:ham}

In addition to computing the chromatic number of a finite graph, a classical NP-complete problem in graph theory is deciding whether a given
connected
graph admits a Hamiltonian cycle. Here, a \emph{Hamiltonian cycle}\index{Hamiltonian cycle}
is a circuit in $\gam$ that visits every vertex of $\gam$ exactly once.
A graph that admits a Hamiltonian cycle is simply called Hamiltonian.
Much like vertex colorability versus edge colorability, there is a notion of a circuit in $\gam$ that traverses every edge exactly once,
called an \emph{Eulerian cycle}\index{Eulerian cycle}.
It is a standard fact that a connected graph admits an Eulerian cycle if and only if each vertex has even degree. Thus, determining whether
a graph admits an Eulerian cycle is a purely local question, and the existence of a Hamiltonian cycle is a global question, impervious to
local methods. We direct the reader to~\cite{diestel-book} for background on Eulerian and Hamiltonian paths and cycles in graphs.

Let $(x_0,\ldots,x_n)$ denote a Hamiltonian cycle in $\gam$, and let \[\{x_0^*,\ldots,x_n^*\}\subset V=H^1(A(\gam),F)\] denote the
corresponding dual classes. Proposition~\ref{prop:coho} implies that \[x_i^*\cup x_{i+1}^*\neq 0\] for all $i$, where the indices are considered
cyclically modulo $n$. This is the fundamental observation when it comes to characterizing Hamiltonicity of $\gam$ in terms of the intrinsic
algebra of $A(\gam)$.

Let $(V,W,q)$ be a triple consisting of a vector space $V$ equipped with a vector-space-valued (i.e.~$W$-valued) (anti)-symmetric bilinear
pairing. We will assume that $V$ is finite dimensional. We say that $(V,W,q)$ is \emph{Hamiltonian}\index{Hamiltonian triple}
if for all bases $\{v_0,\ldots,v_n\}$ of
$V$, there is a permutation $\sigma\in S_{n+1}$ such that for all $i$, we have $q(v_{\sigma(i)},v_{\sigma(i+1)})\neq 0$.

Setting \[V=H^1(A(\gam),F),\quad W=H^2(A(\gam),F),\quad q=\cup,\] suppose that $(V,W,q)$ is Hamiltonian. Then there is a basis
$\{x_0^*,\ldots,x_n^*\}$ consisting of classes dual to the vertices of $\gam$. The Hamiltonicity of the triple immediately implies the
existence of a permutation $\sigma$ such that \[x_{\sigma(i)}^*\cup x_{\sigma(i+1)}^*\neq 0\] for all relevant indices, which immediately
implies that $\gam$ admits a Hamiltonian cycle.

\begin{thm}[See~\cite{FKK2021ham}]\label{thm:hamiltonian}
Let $\gam$ and $(V,W,q)$ be as above. Then $\gam$ admits a Hamiltonian cycle if and only if $(V,W,q)$ is Hamiltonian.
\end{thm}

The reader may check as an easy exercise that the Hamiltonicity of a triple $(V,W,q)$ implies that the triple is in fact pairing-connected,
so that if $(V,W,q)$ is Hamiltonian then $\gam$ is automatically connected by Proposition~\ref{prop:pairing-conn}.

Establishing Theorem~\ref{thm:hamiltonian} is tricky, and requires significantly more insight than Theorem~\ref{thm:color-body}, for instance.
We will attempt to briefly convey the main ideas to the reader in the remainder of this subsection. The reader is directed to
~\cite{FKK2021ham} for a full account.

In order to establish Theorem~\ref{thm:hamiltonian}, it is clearly sufficient to show that if $\gam$ is Hamiltonian then the triple
$(V,W,q)$ is also Hamiltonian. One may begin with the standard dual basis $\{x_0^*,\ldots,x_n^*\}$ for $V$ and hope to bootstrap it to show
that $(V,W,q)$ is Hamiltonian. One can begin with a change of basis matrix $A$, which transforms $\{x_0^*,\ldots,x_n^*\}$ to
a given basis $\{v_0,\ldots,v_n\}$ for $V$. We write $A=(a_i^j)$, where the subscript refers to the row and the superscript refers to the
column of a given entry.

We leave it as an easy exercise for the reader to show the following:

\begin{lem}\label{lem:2-graph}
The triple $(V,W,q)$ is Hamiltonian if for all $A\in\mathrm{GL}_{n+1}(A)$, there is a permutation $\sigma\in S_{n+1}$ such that for all
$0\leq i\leq n$, there exists a $0\leq j\leq n$ such that
\[A_i^j=\begin{pmatrix}a_{\sigma(i)}^j&a_{\sigma(i)}^{j+1}\\a_{\sigma(i+1)}^j& a_{\sigma(i+1)}^{j+1}\end{pmatrix}\] is invertible, where
all indices are considered cyclically.
\end{lem}

Lemma~\ref{lem:2-graph} gives rise to a natural definition that one can associate to matrices (which need not be invertible, or even square).
The \emph{two-row graph}\index{two-row graph} $\GG(A)$ of a matrix $A$ is a graph whose vertices are the rows
$\{r_0,\ldots, r_n\}$ of $A$, and whose columns are given by
the relation $\{r_i,r_j\}\in E(\GG(A))$ if the matrix
\[A_{i,j}^k=\begin{pmatrix}a_{i}^k&a_{i}^{k+1}\\a_{j}^k& a_{j}^{k+1}\end{pmatrix}\] is invertible for some $k$.

It is clear from Lemma~\ref{lem:2-graph} that $(V,W,q)$ is Hamiltonian provided that $\GG(A)$ is itself 
Hamiltonian for all suitable matrices $A$.
To get a feel for $\GG(A)$, the reader is encouraged to prove directly that $\GG(A)$ is connected whenever $A$ is invertible.
The heart of the proof of Theorem~\ref{thm:hamiltonian} is the following:

\begin{lem}\label{lem:hamilton-key}
Let $A$ be an invertible matrix. Then $\GG(A)$ is Hamiltonian.
\end{lem}

Lemma~\ref{lem:hamilton-key} is a curious fact in its own right, and its
proof is fairly involved. Producing a Hamiltonian cycle directly in $\GG(A)$ appears to be a difficult problem itself, and which
has the feel of an NP-complete problem (though this is by no means a theorem). Thus, one needs to use more indirect methods
to find a Hamiltonian cycle in $\GG(A)$.

The key idea is to analyze block submatrices of a matrix $A$
which consist of nonzero entries with one-dimensional row spaces. One can consider maximal such blocks, which give
rise to a partition of the products of entries of $A$ which contribute to the determinant of $A$, according to the standard Leibniz formula.
Using certain symmetries, one can then argue that if no Hamiltonian cycle exists in $\GG(A)$ then all summands in the determinant of $A$
cancel in pairs, and hence the determinant of $A$ is zero.

\subsection{Graph expanders}

In this subsection, we leave behind individual graphs, and concentrate on families of graphs known as graph expanders. Graph expanders
are sequences of connected graphs that are uniformly sparse and uniformly difficult to separate. Expander families find applications in
a myriad of different fields, such as knot theory, spectral graph theory and spectral geometry, probabilistic computation, and network theory.
We direct the reader to~\cite{LZ08,BY13,Bourg09,Alon86,KowalskiBook,LPS88,LubBook,HLWBAMS} for references relevant to this section.

A sequence of finite graphs $\{\gam_i\}_{i\in\bN}$ is called a \emph{graph expander family}\index{graph expander}
if the following conditions are satisfied:

\begin{enumerate}
\item
For all $i$, the graph $\gam_i$ is connected.
\item
There is a $d$ such that for all $i$, the maximum degree of a vertex of $\gam_i$ is at most $d$.
\item
We have $|V(\gam_i)|\longrightarrow\infty$.
\item
The Cheeger constant of $\gam_i$ is uniformly bounded away from zero, independently of $i$.
\end{enumerate}

Here, the \emph{Cheeger constant}\index{Cheeger constant} of a graph
$\gam$ is defined by considering subsets $A\subset V(\gam)$ such that $|A|\leq |V(\gam)|/2$,
and by looking at $\partial A$, which is defined to be the set of of vertices of $V(\gam)\setminus A$ that are adjacent to a vertex of $A$.
The \emph{isoperimetric constant}\index{isoperimetric constant}
of $A$ is defined to be \[c_A=\frac{|\partial A|}{|A|},\] and the Cheeger constant $c$ is the minimum of
$c_A$ as  $A$ ranges over all admissible subsets of $V(\gam)$. From this point of view, it is clear why the Cheeger constant measures
the difficulty in separating $\gam$: in order to completely cut a set $A\subset V(\gam)$ out of $\gam$, one has to sever at least
$c\cdot |A|$ edges.

By associating the standard cohomology triple \[V_i=H^1(A(\gam_i),F),\, W_i=H^2(A(\gam_i),F),\, q_i=\cup,\]
some of the defining properties of graph expanders translate almost immediately. Namely, we have 
$|V(\gam_i)|\longrightarrow\infty$ if and only if
$\dim V_i\longrightarrow\infty$, and $\gam_i$ is connected if and only if $(V_i,W_i,q_i)$ is $q_i$-pairing-connected.

The remaining conditions for defining graph expanders require some new ideas. The degree of a vertex is already characterized
in terms of centralizers via Proposition~\ref{prop:degree} above. Since centralizers of elements are less transparently cohomological
objects, we first translate this notion of degree into linear algebra. Let $(V,W,q)$ be a vector space with a vector-space-valued
bilinear pairing. If $\emptyset\neq S\subset V$ and $B$ is a basis for $V$, we write
\[d_B(S)=\max_{s\in S}|\{b\in B\mid q(s,b)\neq 0\}|.\] To get rid of the dependence on $B$, we set $d(S)$ to be the minimum of $d_B(S)$,
taken over all possible bases. To get rid of the dependence on $S$, we set $d(V)$ to be the minimum of $d(S)$, taken over all $S$ which
span $V$. The quantity $d(V)$ is called the \emph{$q$-valence}\index{$q$-valence} of $V$.

A reader who has understood the ideas in the proof of Theorem~\ref{thm:coho-alg} will have no trouble proving the following fact:

\begin{prop}\label{prop:q-valence}
Let $(V,W,q)$ be the usual cohomological triple associated to $A(\gam)$, and let $d$ be the maximum degree of a vertex of $\gam$.
Then $d(V)=d$.
\end{prop}

It remains to properly define the Cheeger constant of the triple $(V,W,q)$. Again, a reader who has absorbed the ideas
in the proof of Theorem~\ref{thm:coho-alg} could probably guess the definition. Let $Z\subset V$ be a vector space with
$0\neq\dim Z\leq (\dim V)/2$. We will write $C$ for the orthogonal complement of $Z$, which is to say the set of vectors $v\in V$ such that
$q(v,z)=0$ for all $z\in Z$. The isoperimetric constant of $Z$ is defined to be \[c_Z=\frac{\dim V-\dim Z-\dim C+\dim(C\cap Z)}{\dim Z}.\]
The Cheeger constant $c_V$ of the triple $(V,W,q)$ is taken to be the infimum of $c_Z$ as $Z$ varies over all nonzero subspaces of $V$ of
dimension at most half of that of $V$.

Let $\{x_1,\ldots,x_n\}$ denote the vertices of $\gam$ and $\{x_1^*,\ldots,x_n^*\}$ be the dual generators of $H^1(A(\gam),F)$.
If $B\subset \{x_1,\ldots,x_n\}$, write $B^*$ for the corresponding subset of $\{x_1^*,\ldots,x_n^*\}$.
The following is an exercise for the reader:

\begin{prop}
Let $\emptyset\neq B\subset V$, and let $Z\subset V$ be generated by $B^*$. Then \[c_Z=\frac{|\partial B|}{|B|}.\]
\end{prop}

Thus, the Cheeger constant of $\gam$ is bounded below by $c_V$. \emph{A priori}, there are many more subspaces of $V$ than there are
subgraphs of $\gam$, so that in principle $c_V$ could be strictly smaller than the Cheeger constant of $\gam$.

\begin{thm}[See~\cite{FKK2020exp}]\label{thm:expander}
Let $\{\gam_i\}_{i\in\bN}$ be a sequence of graphs, and let \[\{(V_i,W_i,q_i)\}_{i\in\bN}\] be the corresponding cohomological triples.
We have that $\{\gam_i\}_{i\in\bN}$ forms a family of expanders if and only if:
\begin{enumerate}
\item
For each $i$, the triple $(V_i,W_i,q_i)$ is pairing-connected.
\item
There is a $d$ such that the $q_i$-valence of $V_i$ is bounded above by $d$.
\item
We have $\dim V_i\longrightarrow\infty$.
\item
There is an $\epsilon>0$ such that for all $i$, we have $c_{V_i}\geq\epsilon$.
\end{enumerate}
\end{thm}

An abstract sequence of triples $\{(V_i,W_i,q_i)\}_{i\in\bN}$ is called a family of
\emph{vector space expanders}\index{vector space expander}
(not to be confused with \emph{dimensional expanders}\index{dimensional expander}, cf.~\cite{LZ08,BY13,Bourg09}).
In light of the preceding discussion, in order to establish Theorem~\ref{thm:expander}, it suffices to show that
for each $i$, the Cheeger constant $c_{V_i}$ coincides with
the Cheeger constant of $\gam_i$. Unfortunately, the author does not know a conceptually simple proof of this fact. The proof given
in~\cite{FKK2020exp} involves a rather technical sorting argument, and so we will not comment on it any further.

\subsection{Graph automorphisms}\label{ss:aut}

One of the most basic questions one can ask about a graph (and indeed about a relation) is how symmetric it is. Symmetry is measured
by the richness of the automorphism group, and the smaller the size of the automorphism group, however it is measured, the less symmetric
the object.

The automorphisms of graphs are of great interest in graph theory
~\cite{GR2001-book,diestel-book,BW04-book}, and in complexity theory as well~\cite{babai-proceedings}.
Many finite graphs are highly symmetric. For instance, the automorphism group of a $k$-clique is the full symmetric group on $k$ letters.
Many other graphs have no nontrivial automorphisms. For instance, take a path of length five, with vertices labeled linearly as
$\{a,b,c,d,e,f\}$,
and add another vertex $g$ which is adjacent only to $d$. The resulting graph $\gam$ has no nontrivial automorphisms, as is readily
verified by an exhaustive check. See Figure~\ref{f:gam}.

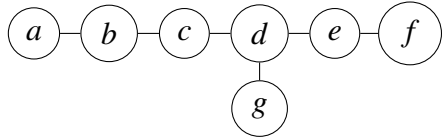
\begin{figure}[h!]
  \tikzstyle {bv}=[black,draw,shape=circle,fill=black,inner sep=1pt]
\begin{tikzpicture}[main/.style = {draw, circle}]
\node[main] (1) {$a$};
\node[main] (2) [right of=1] {$b$};
\node[main] (3) [right of=2] {$c$}; 
\node[main] (4) [right of=3] {$d$};
\node[main] (5) [right of=4] {$e$}; 
\node[main] (6) [right of=5] {$f$};
\node[main] (7) [below of=4] {$g$};

\draw (1)--(2)--(3)--(4)--(5)--(6);
\draw (4)--(7);
\end{tikzpicture}%
\caption{A graph that has no nontrivial automorphisms.}
\label{f:gam}
\end{figure}

Observe that a nontrivial automorphism of a graph $\gam$ gives rise to a non-inner automorphism of $A(\gam)$. Moreover, if $v\in V(\gam)$,
then the function $v\mapsto v^{-1}$ extends to a non-inner automorphism of $A(\gam)$ via the identity on the remaining vertices.
It is easy to see that the group $\Aut(A(\gam))$,
and in fact $\Out(A(\gam))$, contains a subgroup isomorphic to \[\Aut(\gam)\ltimes(\bZ/2\bZ)^{|V(\gam)|}.\] Thus, if $\gam$ admits a nontrivial
automorphism, then $\Out(A(\gam))$ contains a nonabelian finite subgroup.

\begin{thm}[See~\cite{FKK2019}]\label{thm:aut}
Let $\gam$ be a finite simplicial graph. We have that $\gam$ admits a nontrivial automorphism if and only if $\Out(A(\gam))$ contains a finite
nonabelian subgroup.
\end{thm}

The ``only if" direction follows from the discussion preceding Theorem~\ref{thm:aut}. The converse is significantly harder and requires a more
careful analysis of $\Out(A(\gam))$.

A result of M.~Laurence (\cite{Laurence95}, cf.~\cite{Servatius1989})
says that $\Aut(A(\gam))$ is generated by automorphisms of the following type.
\begin{enumerate}
\item
Vertex inversions.
\item
Graph automorphisms.
\item
Partial conjugations.
\item
Dominated transvections.
\end{enumerate}

Graph automorphisms have already been discussed, and \emph{vertex inversions}\index{vertex inversion}
have been mentioned above as arising from the map
$v\mapsto v^{-1}$ for some $v\in V(\gam)$. A \emph{partial conjugation}\index{partial conjugation}
is given by considering a vertex $v\in V(\gam)$ whose star $\st(v)$
separates $\gam$. The automorphism acts by conjugation by $v$ on one component of $\gam\setminus\st(v)$ and by the identity on
the remaining components of $\gam$.

To define~\emph{dominated transvections}\index{dominated transvection}, we say that a vertex $v\in V(\gam)$ \emph{dominates}
a vertex $w\in V(\gam)$ if $\lk(w)\subset\st(v)$.
Then, the map $w\mapsto wv$ extends to an automorphism of $A(\gam)$ via the identity on the remaining vertices. Domination 
is clearly a relation
on vertices of $\gam$ that can be determined from the combinatorics of $\gam$.

\begin{proof}[Sketch of proof of Theorem~\ref{thm:aut}]
We suppose that $\gam$ admits no nontrivial automorphisms. A theorem of Toinet
~\cite{Toinet13} implies that if $\phi\in\Aut(A(\gam))$ has finite order
then $\phi$ acts nontrivially on $H_1(A(\gam),\bZ)\cong \bZ^{|V(\gam)|}$. Thus, it suffices to consider the action of automorphisms
on $H_1(A(\gam),\bZ)$, and the effect of partial conjugations is then trivial.

Next, one shows that if there is a cycle $\{v_1,v_2,\ldots,v_k,v_1\}_{k\geq 2}$
where $v_i$ dominates $v_{i+1}$ (with the indices considered cyclically),
then $\gam$ admits a nontrivial automorphism, specifically an automorphism that exchanges two vertices of $\gam$.
It follows that no such cycles exist. We may therefore order the vertices of $\gam$ in such
a way that if $v_i<v_j$ then $v_j$ cannot dominate $v_i$. If we then write the image of $\Out(A(\gam))$ in $\mathrm{GL}_n(\bZ)$
with respect to the corresponding ordered basis for $H_1(A(\gam),\bZ)$, the result is a group of upper triangular integer matrices. Such
a group has only abelian finite subgroups (coming from diagonal matrices with entries $\pm 1$). Thus, if $\gam$ has no automorphisms
then $\Out(A(\gam))$ has only abelian finite subgroups.
\end{proof}

We remark that in the proof of Theorem~\ref{thm:aut}, one of the key observations is that a graph with a domination cycle admits a nontrivial
automorphism. The converse of this statement is false. The $5$-cycle $C_5$ admits many automorphisms, but no two vertices
dominate each other.

\subsection{Some further entries in the combinatorics--algebra dictionary}\label{ss:further}
There are a number of other results relating the combinatorics of graphs to the algebraic structure of groups which we will not
discuss in detail for the sake of space. We briefly mention two results appearing in~\cite{HermillerSunic}. Recall that a group $G$
is \emph{poly-free}\index{poly-free group} if there is a finite length 
subnormal filtration of $G$ by subgroups such that successive quotients are free.
Hermiller--\v{S}uni\'c proved that a right-angled Artin group is always poly-free, and that the length of the poly-free filtration is bounded
above by the chromatic number of the defining graph. In the same paper, they established that $A(\gam)$ is a semidirect product
of two finitely generated free groups if and only if $\gam$ is a tree or a \emph{complete bipartite graph}\index{complete bipartite graph},
which is to say a join of two completely disconnected graphs. Moreover, for a connected graph $\gam$ with at least two vertices,
the poly-free length of $A(\gam)$ is exactly two if and only
if there is a subset $D\subset V(\gam)$ such that no pair of elements of $D$ spans an edge, and every circuit in $\gam$ meets
$D$ in at least two vertices. It is an interesting direction for future research to investigate the relationship between the normal structure
of $A(\gam)$ and the combinatorics of $\gam$, and it appears that this subject is largely unexplored.

\subsection{Usefulness beyond group theory and combinatorics}\label{ss:applications}

The various correspondences between combinatorics of graphs and algebraic structures of groups have theoretical and practical applications
beyond the structural framework of Question~\ref{q:comb-intro} and its refinements. Here, we record some specific examples.

\subsubsection{Complexity of problems in combinatorial group theory}

One of the main applications of the foregoing discussion is in the domain of complexity theory, which is hardly surprising in light of the
fact that many computationally difficult problems (i.e.~NP-complete problems\index{NP-complete},
cf.~\cite{Minsky67,GJ1979,AB09}) are formulated in a finitistic way, with reference to only
combinatorial structures.

Consider a right-angled Artin group $A(\gam)$, and a homomorphism 
\[\phi\colon A(\gam)\longrightarrow F_{n_1}\times F_{n_2}\times F_{n_3},\] where
$F_{n_i}$ denotes a free group of rank $n_i$, and where \[n_1+n_2+n_3=|V(\gam)|.\] If $\gam$ is specified (e.g.~by a list of vertices and
pairs of adjacent vertices) and $\phi$ is specified in terms of
the image of each vertex of $\gam$ with respect to a fixed free basis of each of the free group factors in the target of $\phi$, then it is easy to
check if $\phi$ is a homomorphism that is surjective on the level of first rational homology.
Indeed, it suffices to check first that $\phi$ is well-defined, meaning that adjacent
vertices in $\gam$ are sent to commuting elements of $F_{n_1}\times F_{n_2}\times F_{n_3}$, which can be performed efficiently.
The latter claim results from the fact that centralizers of elements in $F_{n_1}\times F_{n_2}\times F_{n_3}$ are straightforward to
describe, and because the word problem is efficiently solvable. Then, one must check that $\phi$ is surjective on the level
of first rational homology, which is an easy linear algebra problem. In light of Theorem~\ref{thm:color-body}, the data specifying the
homomorphism $\phi$ forms a (short) certificate of the fact that $\gam$ is $3$-colorable. Since the $3$-colorability of $\gam$
and the existence of this homomorphism are equivalent, the problem of deciding whether such a homomorphism exists
is NP-complete. To state this conclusion formally:

\begin{prop}\label{prop:np-complete}
Let $\gam$ be a finite graph with $|V(\gam)|=N$, and let \[\{F_{n_1},\ldots,F_{n_k}\}\] be free groups such that 
\[\sum_{i=1}^k n_i=N.\] Write $G=\prod_i F_{n_i}$.
\begin{enumerate}
\item
If $k=2$ then the problem of
 deciding whether or not there exists a homomorphism $A(\gam)\longrightarrow G$ that is surjective on first rational homology is in P.
\item
If $k=3$ then the problem of
 deciding whether or not there exists a homomorphism $A(\gam)\longrightarrow G$ that is surjective on first rational homology is NP-complete.
 \item
 The problem of finding the minimal $k$ for which there exist free groups $\{F_{n_1},\ldots,F_{n_k}\}$ as above and a
 homomorphism $A(\gam)\longrightarrow G$ that is surjective on first rational homology is NP-complete.
\end{enumerate}
\end{prop}

Finding explicit examples of NP-complete problems is always of interest in complexity theory, and given the profusion of them in
graph theory, Proposition~\ref{prop:np-complete} is just a taste of the richness of the available theory arising in the context of groups.

\subsubsection{Hamiltonicity testing}

Continuing in the theme of NP-complete problems, it is well-known that deciding if a finite graph admits a Hamiltonian path or Hamiltonian
cycle is NP-complete, as we have mentioned above. The ideas surrounding Theorem~\ref{thm:hamiltonian} can be used to certify
that certain graphs are not Hamiltonian, in a purely finitistic linear algebraic way.

To expand on this a bit, first note that the field over which cohomology is considered is arbitrary. In particular, we may assume that
the underlying field is just the field with two elements. Under this assumption, all the relevant vector spaces become finite sets, and are
hence amenable to combinatorial techniques.

Consider then the standard cohomological triple $(V,W,q)$ for a right-angled Artin group
$A(\gam)$. In order to show that $\gam$ is not Hamiltonian,
it suffices to find a single basis for $V$ which witnesses the claim that $(V,W,q)$ is not Hamiltonian. Thus, such a basis can be used
as a short certificate that a graph contains no Hamiltonian circuit.

\subsubsection{Linear algebraic detection of graph expanders}

Considering cohomology with coefficients in a field with two elements allows a finitary and algebraic way to check if a sequence of
graphs is a family of expanders. Moreover, it is shown in~\cite{FKK2020exp}
that there are families of vector space expanders that do not arise from
the cohomology of families of graph expanders. Thus, the theory of vector space expanders is \emph{a priori} richer than the
theory of graph expanders. Some practical applications of expanders can be found in~\cite{CLG,GILVZ}, for instance.

\subsubsection{Interactive proof systems}

Many interactive proof systems function as a way for a prover to demonstrate a proposition to a skeptical verifier. Using an unbiased
random bit sent by the verifier, the prover sends a response that is conditioned on the value
of the random bit. In this way, the verifier's ignorance of the prover's 
private information is balanced by the prover's ignorance of the value of the
bit that will be sent by verifier, and this balance can be used to communicate the existence of knowledge without revealing its content.
This is, for instance, the idea behind \emph{zero-knowledge proof protocols}\index{interactive proof},
in which the prover holds a certificate
for an instance of an NP-complete problem, and convinces the verifier of the fact that she is in possession of a valid certificate
without revealing the certificate itself.
Any NP-complete problem can be used as a platform. Thus, linear algebraic versions of
Hamiltonicity as in Theorem~\ref{thm:hamiltonian} and Proposition~\ref{prop:np-complete} are suitable for formulating a zero-knowledge
proof protocol. A detailed explanation of a platform using Theorem~\ref{thm:hamiltonian} is given in~\cite{FKK2021ham}.
For general background
on interactive proofs and zero-knowledge proof protocols, we refer the reader to~\cite{Blum87,rosen-book,AB09,GMW91,BabaiMoran}.

\subsubsection{Group-based cryptosystems}

Many cryptosystems rely on computational problems that are difficult to solve directly, which is why many modern cryptographic
protocols assume $P\neq NP$. The theme of this section has been the translation of combinatorial properties of graphs, and especially
computationally interesting ones, into algebraic language. This immediately suggests numerous potential
group-based cryptosystems\index{group-based cryptography},
a topic which has been developing rapidly in recent decades. Explicit cryptosystems using right-angled Artin groups as a platform have
been proposed in~\cite{FK2016}, for example. Translating the graph homomorphism problem (which is NP-complete) into
an instance of the subgroup homomorphism problem for right-angled Artin groups, one can formulate a secure
authentication scheme, for instance. For further discussion of specific cryptosystems and for a biased sample of the literature,
we direct the reader to
~\cite{FKK2020-pub,FKK2019,MSU08,KoLee00,KahShp15,FKR12}.

\section{The extension graph and its properties}

We now leave the world of the finite graph $\gam$ and its relationship with $A(\gam)$, and turn to the (usually) infinite extension graph
$\gam^e$. We recall that $\gam^e$ is a development of $\gam$ into a graph on which $A(\gam)$ acts by conjugation. So, we fix an
identification of the vertices of $\gam$ with generators for $A(\gam)$, set the vertices of $\gam^e$ to be the collection of all conjugates
of $V(\gam)$ by elements of $A(\gam)$, and set the edge relation to be commutation inside of $A(\gam)$. The reader will find that
the ideas here, though still fundamentally relating combinatorics to algebra, are quite different from those in Section~\ref{sec:comb}.

\subsection{Basic properties of the extension graph}

Some properties of the extension graph are easy to prove. For instance:

\begin{prop}
The extension graph $\gam^e$ is finite if and only if $\gam$ is complete.
\end{prop}

Others are somewhat less obvious. We note some which will be useful in the sequel, and which otherwise will give the reader a better
idea of how the extension graph functions.

\begin{prop}[See~\cite{KK2013}]\label{prop:gex-basic}
The extension graph $\gam^e$ enjoys the following properties:
\begin{enumerate}
\item
The graph $\gam^e$ is connected if and only if $\gam$ is connected.
\item
The graph $\gam^e$ is connected and of infinite diameter if and only if $\gam$ is connected, has at least two vertices, and is not a join.
\item
The size of a maximal clique in $\gam$ and $\gam^e$ coincide.
\item
If $\Lambda$ is a subgraph of $\gam$ then $\Lambda^e$ is a subgraph of $\gam^e$.
\item
The graph $\gam^e$ is $k$-colorable if and only if $\gam$ is $k$-colorable.
\end{enumerate}
\end{prop}

The proof of item (2) of Proposition~\ref{prop:gex-basic}
we will provide probably illustrates the diversity of methods that can be used in investigating right-angled Artin
groups.

\begin{proof}[Sketch of proof of Proposition~\ref{prop:gex-basic}, (2)]
Consider a collection of disjoint compact annuli $\{A_v\mid v\in V(\gam)\}$, one
for each vertex of $\gam$. Glue two such annuli $A_v$ and $A_w$ together along a disk if the vertices
$v$ and $w$ are \emph{not} adjacent in $\gam$. We do this
in such a way that the result is an orientable surface $\Sigma$
 with boundary. A key observation is that since $\gam$ is not a join, its complement
graph $X$ is connected. Therefore, $\Sigma$ is a connected surface. Since $\Sigma$ was built out of at least two annuli, an easy
Euler characteristic computation shows that $\Sigma$ is of hyperbolic type (i.e.~admits a complete hyperbolic metric of finite volume).
We will name the core curves of the annuli in the construction $\{\gamma_1,\ldots,\gamma_n\}$.

The (isotopy class of the) homeomorphism of $\Sigma$ given by cutting $\Sigma$ open along $\gamma_i$ and re-gluing with a full
right-handed twist is called a (right-handed) \emph{Dehn twist}\index{Dehn twist} about $\gamma_i$, and is denoted by $T_i$.
Recall that the group of isotopy classes of 
(orientation preserving) homeomorphisms of $\Sigma$ is called the \emph{mapping class group}\index{mapping class group}
of $\Sigma$, and is
written $\Mod(\Sigma)$~\cite{FM2012}. A result of the author~\cite{Koberda2012}
shows that there is an $N>0$ such that for all $k\geq N$, the subgroup of $\Mod(\Sigma)$ generated by
$\{T_1^k,\ldots,T_n^k\}$ is isomorphic to $A(\gam)$.

The surface $\Sigma$ has an associated \emph{curve graph}\index{curve graph}
$\CC(\Sigma)$, which is of infinite diameter. This curve graph consists of isotopy classes
of embedded, essential, nonperipheral loops on $\Sigma$, with the edge relation being disjoint realization. There are certain mapping
classes $\psi$ which have the property that for any vertex $c$ of $\CC(\Sigma)$, the graph distance between $c$ and $\psi^k(c)$ tends
to infinity as $k$ tends to infinity~\cite{Thurston-bull88}.
These mapping classes are called \emph{pseudo-Anosov}\index{pseudo-Anosov mapping class}, and are typical inside of $\Mod(\Sigma)$.

In particular, realizing $A(\gam)<\Mod(\Sigma)$ as above, there is an element $g\in A(\gam)$ whose realization as a mapping class is
pseudo-Anosov. Moreover, the realization $A(\gam)<\Mod(\Sigma)$ is compatible with a realization of $\gam^e\subset\CC(\Sigma)$.
Specifically, if $v\in V(\gam)$ is associated to a Dehn twist about $\gamma_i$ and if $h\in A(\gam)$ corresponds to the mapping class
group $\psi_h$, then the vertex $v^h$ is sent to $\psi_h(\gamma_i)$.

Now, since we have a map $\gam^e\longrightarrow\CC(\Sigma)$ which respects the edge relation, general facts about graph homomorphisms
 imply that it cannot be
distance increasing. Thus, if $d_{\CC(\Sigma)}(\gamma_i,\psi_h^k(\gamma_i))$ tends to infinity then $d_{\gam^e}(v,v^{h^k})$ also
tends to infinity. The conclusion now follows.
\end{proof}

It turns out that mapping class groups of surfaces are extremely useful tools for probing right-angled Artin groups, and that many of their
properties can be paired analogously. This is a theme that will recur in this section, and we will comment more on it below.

\subsection{The extension graph and subgroups}

One useful property of the extension graph, and for which it was developed in the first place, is that the extension graph classifies
right-angled Artin subgroups of of a right-angled Artin group. Classically, we know that subgroups of finitely generated free abelian groups
are again free abelian (by the classification of finitely generated modules over a principal ideal domain)
 and subgroups of free groups are always free (by the Nielsen--Schreier Theorem).
Since right-angled Artin groups interpolate
between these two extremes, it is therefore a natural question whether
(finitely generated) subgroups of right-angled Artin groups are again right-angled Artin
groups, and if so what sorts of right-angled Artin groups they are.

It is not true that subgroups of right-angled Artin groups are again right-angled Artin groups. There are many different subgroups of
right-angled Artin groups, ranging from surface groups
~\cite{SDS1989,CW2004} to hyperbolic $3$-manifold groups~\cite{Agol2008,Agol2013,Wise2009,Wise2011,Wise2012}
 to many arithmetic lattices in rank one Lie groups~\cite{BHW2011},
all the way to groups with various exotic finiteness properties~\cite{BestBrady97}.
It is in fact known that every finitely generated subgroup of $A(\gam)$ is again a right-angled Artin group if and only if $\gam$ has no
subgraph isomorphic to a square or to a path of length three, by a result of Droms~\cite{Droms1987-sub}.

It is difficult to characterize all subgroups of right-angled Artin groups, even finitely presented ones (see~\cite{Brid-MRL}).
Some general known facts are that every nonabelian subgroup
of a right-angled Artin group contains a nonabelian free group by a result of Baudisch~\cite{baudisch},
and in fact any such subgroup surjects to a nonabelian free group by a result of Antol\'in--Minasyan~\cite{ant-min}. A nonabelian subgroup
of a right-angled Artin group must surject to $\bZ^2$~\cite{duch-krob,Koberda-survey}. Solvable subgroups of right-angled Artin groups
are automatically finitely generated and free abelian, by the Flat Torus Theorem~\cite{BH1999}.

Given the difficulty of understanding general subgroups of right-angled Artin groups,
it is therefore interesting and natural to wonder which subgroups of $A(\gam)$ are of the form $A(\Lambda)$, and what sorts of 
graphs $\Lambda$
can occur. To the author's knowledge, there is no clean, complete answer available, though the partial answers are satisfying and useful
for many applications.

\begin{thm}\label{thm:gex1}
Let $\Lambda<\gam^e$ be a finite subgraph. Then there is an injective homomorphism $A(\Lambda)\longrightarrow A(\gam)$.
\end{thm}

The injection in Theorem~\ref{thm:gex1} is quite explicit; one simply views vertices of $\Lambda$ as elements in $A(\gam)$ and passes
to a sufficiently high power. Theorem~\ref{thm:gex1} first appeared in a paper of Kim and the author
~\cite{KK2013}, though apparently this fact
was already known to experts in combinatorial group theory. One approach to proving Theorem~\ref{thm:gex1} does not require ideas beyond
those that go into item (2) of Proposition~\ref{prop:gex-basic}. Once the extension graph has been embedded in the curve graph in a way
that preserves both adjacency and non-adjacency, the author's result from
~\cite{Koberda2012} about powers of mapping classes applies and gives
the desired result.

Unfortunately, Theorem~\ref{thm:gex1} does not admit an easy converse. The first examples disproving the
obvious na\"ive converse appeared
in the work of Casals-Ruiz--Duncan--Kazachkov~\cite{CDK2013},
and a large class of examples was produced by Kim and the author~\cite{KK2015GT}. With some
further assumptions on $\gam$, one can formulate a converse to Theorem~\ref{thm:gex1}.

\begin{thm}[See~\cite{KK2013}]\label{thm:gex2}
Suppose $\gam$ has no $3$-cliques, and suppose that $A(\Lambda)<A(\gam)$. Then $\Lambda$ is a subgraph of $\gam^e$.
\end{thm}

Theorem~\ref{thm:gex2} is a corollary of a more general result, which is the most general converse to Theorem~\ref{thm:gex1} that is
known to the author.

\begin{thm}[See~\cite{KK2013}]\label{thm:gex3}
Suppose that $A(\Lambda)<A(\gam)$. Then $\Lambda$ is a subgraph of the clique graph $(\gam^e)_k$.
\end{thm}

The basic idea behind Theorem~\ref{thm:gex3} is again to use mapping class groups, though it is significantly more complicated than
Theorem~\ref{thm:gex1} and Proposition~\ref{prop:gex-basic}. One builds certain ``partial" pseudo-Anosov mapping classes in
the image of $A(\Lambda)$ and builds an embedding of a larger graph $X$ into $\gam^e$, which contains $\Lambda$ in its clique
graph. Incidentally, Theorem~\ref{thm:gex3} has a natural analogue for mapping class groups: if a right-angled Artin group $A(\gam)$
embeds in a mapping class group $\Mod(\Sigma)$, then $\gam$ embeds as a subgraph of $\CC(\Sigma)_k$,
the clique graph of the curve graph (\cite{KK2014IMRN}, cf.~\cite{KKOsaka}). We will avoid giving further details here.

Theorem~\ref{thm:gex3} admits several other corollaries that can serve as converses to Theorem~\ref{thm:gex1}, and also
allows one to prove many results that relate the combinatorics of $\gam$ to the structure of $A(\gam)$. Given the conclusion of
Theorem~\ref{thm:gex3}, we leave the following result (originally due to Kambites~\cite{Kambites2009}, who offered a combinatorial
argument that is very different in flavor from the ideas expounded here) as an exercise for the reader:

\begin{prop}\label{prop:square}
Let $\gam$ be a finite graph. Then $\gam$ contains a square if and only if $F_2\times F_2<A(\gam)$.
\end{prop}

Here, by \emph{square}\index{square graph} we mean a graph with four vertices and a cyclic adjacency relation.

\begin{figure}[h!]
  \tikzstyle {bv}=[black,draw,shape=circle,fill=black,inner sep=1pt]
\begin{tikzpicture}[main/.style = {draw, circle}]
\node[main] (1) {$a$};
\node[main] (2) [right of=1] {$b$};
\node[main] (3) [below of=1] {$c$}; 
\node[main] (4) [below of=2] {$d$};

\draw (1)--(2);
\draw (3)--(4);
\draw (2)--(4);
\draw (1)--(3);
\end{tikzpicture}%
\caption{The square.}
\label{f:square}
\end{figure}
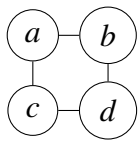

\subsection{A characterization of cographs via right-angled Artin groups and the geometry of the extension graph}

An important class of graphs that occurs naturally in graph theory is the class of \emph{cographs}\index{cograph}, or $P_4$-free graphs
(see~\cite{cograph1,cograph2,cograph3} for some early references introducing cographs).
These are simply the graphs that do not have the path $P_4$ of length three as a subgraph.

\begin{figure}[h!]
  \tikzstyle {bv}=[black,draw,shape=circle,fill=black,inner sep=1pt]
\begin{tikzpicture}[main/.style = {draw, circle}]
\node[main] (1) {$a$};
\node[main] (2) [right of=1] {$b$};
\node[main] (3) [right of=2] {$c$}; 
\node[main] (4) [right of=3] {$d$};

\draw (1)--(2)--(3)--(4);
\end{tikzpicture}%
\caption{The graph $P_4$.}
\label{f:p4}
\end{figure}
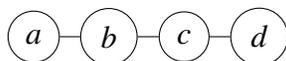

Right-angled Artin groups on cographs can be characterized algebraically, and right-angled Artin groups provide a perspective on cographs
that insight into one of their most fundamental properties, i.e.~recursive definition.

\begin{thm}[\cite{KK2013}]\label{thm:cograph}
Let $\gam$ be a finite connected graph.
The following are equivalent:
\begin{enumerate}
\item
The graph $\gam$ has no (full) subgraph isomorphic to $P_4$.
\item
The graph $\gam^e$ has no (full) subgraph isomorphic to $P_4$.
\item
The graph $\gam$ is either a single vertex or splits as a nontrivial join.
\end{enumerate}
\end{thm}

\begin{cor}\label{cor:cograph}
The graph $\gam$ is a cograph if and only if $A(\gam)$ does not contain a copy of $A(P_4)$.
\end{cor}

In particular, Theorem~\ref{thm:cograph} shows that a right-angled Artin group cannot contain ``hidden" copies of $A(P_4)$. If $A(\gam)$
contains $A(P_4)$ then one can decide simply from looking at the graph $\gam$. This is in contrast to other classes of graphs. For instance,
$A(P_4)$ contains a copy of $A(P_5)$, where $P_5$ denotes the path of length four. Thus, there can be hidden copies of $A(P_5)$.
For a more striking example, one may consider $X_6$, the complement graph of the hexagon, also known as the triangular prism.
This graph contains no cycle $C_5$ of length $5$, though Kim proved that $A(C_5)<A(X_6)$~\cite{Kim2008}; also, $C_5$ is a subgraph of
the extension graph $X_6^e$, and so Kim's result follows from Theorem~\ref{thm:gex1}.

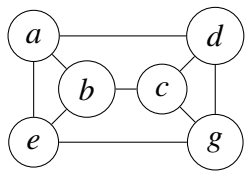
\begin{figure}[h!]
  \tikzstyle {bv}=[black,draw,shape=circle,fill=black,inner sep=1pt]
\begin{tikzpicture}[main/.style = {draw, circle}]
\node[main] (1) {$a$};
\node[main] (2) [below right of=1] {$b$};
\node[main] (3) [right of=2] {$c$}; 
\node[main] (4) [above right of=3] {$d$};
\node[main] (5) [below left of=2] {$e$};
\node[main] (6) [below right of=3] {$g$};

\draw (1)--(2);
\draw (5)--(2);
\draw (1)--(5);
\draw (2)--(3);
\draw (3)--(4);
\draw (3)--(6);
\draw (1)--(4);
\draw (5)--(6);
\draw (6)--(4);
\end{tikzpicture}%
\caption{The graph $X_6$.}
\label{f:x6}
\end{figure}

We leave the proofs of Theorem~\ref{thm:cograph} and Corollary~\ref{cor:cograph}
as an exercise for the reader, as they follow from Proposition~\ref{prop:gex-basic}
and some elementary combinatorial group theory considerations.

Let $\gam$ be a connected graph such that $\gam^e$ has finite diameter. By Theorem~\ref{thm:cograph} (or even just by
Proposition~\ref{prop:gex-basic}), the graph $\gam$ splits as a nontrivial join. If $\gam$ is a cograph and $\Lambda$ is a join
factor of $\gam$, then $\Lambda$ must also be a cograph and hence $\Lambda^e$ also has finite diameter, whence it follows
that $\Lambda$ must also split as a nontrivial join.

Let $\KK_0$ denote a singleton vertex. For $i>0$, we set $\KK_{2i-1}$ to be the collection of all finite graphs obtained as
(possibly trivial) joins
of elements of $\KK_{2i-2}$. We set $\KK_{2i}$ to be the collection of all finite graphs obtained as disjoint unions of elements of
$\KK_{2i-1}$. Clearly for $i\leq j$ we have $\KK_i\subset\KK_j$, and we set \[\KK=\bigcup_{i\geq 0}\KK_i.\]
Clearly, if $\gam\in\KK$ then $\gam$ is a cograph. Conversely, the preceding remarks
and an easy induction on $|V(\gam)|$ show that if $\gam$ is a cograph then $\gam\in\KK$. This coincides with the recursive description
of cographs.

Since $\KK$ is built up recursively, we can give the following characterization of $A(\gam)$ for $\gam\in\KK$, which results immediately from
the preceding discussion:

\begin{cor}
We have $\gam\in\KK$ if and only if $A(\gam)$ is an element of the smallest class of groups that:
\begin{enumerate}
\item
Contains $\bZ$;
\item
Is closed under finite direct products;
\item
Is closed under finite free products.
\end{enumerate}
\end{cor}

For example, note that $\gam\in \KK_0$ if and only if $A(\gam)\cong \bZ$.
We have $\gam\in \KK_1$ if and only if $A(\gam)\cong \bZ^n$ for some
$n$. We have $\gam\in \KK_2$ if and only if $A(\gam)$ is a free product of free abelian groups. A graph $\gam$ lies in $\KK_3$ if and only
if $A(\gam)$ is a direct product of free products of free abelian groups. The following characterizes $\KK_i$ for $i\leq 3$:

\begin{prop}[See~\cite{KK2018JT}]
A graph $\gam$ lies in $\KK_i$ for $i\leq 3$ if and only if $A(\gam)$ has no subgroup isomorphic to $(F_2\times\bZ)*\bZ$.
\end{prop}

As an aside, we note that the hierarchy $\KK$ and the associated right-angled Artin groups is closely related to the theory of right-angled
Artin group actions on the interval and on the circle. It turns out that $A(P_4)$ does not act faithfully by $C^2$ diffeomorphisms on $I$
or $S^1$
~\cite{BKK2019JEMS},
so any right-angled Artin group admitting such an action must have its underlying graph in $\KK$. By a result of Kim and the
author~\cite{KK2018JT}, a right-angled Artin group $A(\gam)$ admits a faithful $C^2$ action on $I$ or $S^1$
if and only if it admits a faithful $C^{\infty}$
such action, if and only if $\gam\in\KK_3$.

\subsection{More on the geometry of the extension graph}

As we have suggested in this section, and in particular in the discussion about Proposition~\ref{prop:gex-basic}, the extension graph
of $\gam$ plays a role analogous to that of the curve graph $\CC(\Sigma)$ of a surface, with the role of the mapping class group
in the latter context played by the group $A(\gam)$ in the former context.

The graph $\CC(\Sigma)$ is very complicated in both its local and its global structure. One of the most important foundational results
about the global structure of $\CC(\Sigma)$ is a result of Masur and Minsky which asserts that $\CC(\Sigma)$ is 
\emph{$\delta$-hyperbolic}\index{$\delta$-hyperbolic}~\cite{MM1999}, see also~\cite{Gromov1987,GdlH88}.
That is, there is a $\delta\geq 0$ so that in any geodesic triangle in $\CC(\Sigma)$, a $\delta$-neighborhood of two of the sides of the
triangle contains the third.

Perhaps the easiest example of an infinite diameter $\delta$-hyperbolic metric space is an infinite diameter tree, which is $0$-hyperbolic.
There are many other examples of $\delta$-hyperbolic spaces that are not trees, such the usual hyperbolic spaces. For most surfaces,
the curve graph $\CC(\Sigma)$ is far from being a tree; it has one end, whereas
for example a locally finite tree that admits a proper and cocompact action by an infinite group
will have at least two ends, as follows from Bass--Serre Theory~\cite{Serre1977}.

The geometry of the extension graph is something in between the curve graph and a tree. To state a precise result, we need the notion
of a quasi-isometry. Let $f\colon X\longrightarrow Y$ be a function between metric spaces. 
Then we say that $f$ is a \emph{quasi-isometry}\index{quasi-isometry}
if there are constants
$\lambda\geq 1$ and $C\geq 0$ such that for all $x,z\in X$, we have
\[\frac{1}{\lambda}\cdot d_X(x,z)-C\leq d_Y(f(x),f(z))\leq \lambda \cdot d_X(x,z)+C,\]
and where for all $y\in Y$ there exists an $x\in X$ such that
\[d_Y(f(x),y)\leq C.\] Here, the distance functions are all interpreted in the relevant spaces.
A quasi-isometry can be thought of a function that is bi-Lipschitz on a large scale. For instance, the integers equipped
with the metric induced from
the real line are quasi-isometric to the real line, and any two finite-diameter metric spaces are quasi-isometric to each other, but
an infinite-diameter metric space is not quasi-isometric to a finite-diameter metric space.

The relation induced by quasi-isometry is
an equivalence relation on metric spaces, and so one often speaks of the quasi-isometry class of a metric space. The quasi-isometry
class of a finitely generated group is the quasi-isometry class of its Cayley graph, equipped with the graph metric;
see~\cite{dlHarpe2000} for more details, for example.

In coarse geometry, one often searches for properties of metric spaces that are invariant under quasi-isometry. Examples of such properties
include $\delta$-hyperbolicity and the number of ends.

A metric space is called a \emph{quasi-tree}\index{quasi-tree}
if it contains a $0$-hyperbolic metric space in its quasi-isometry class. Whereas simplicial
trees are $0$-hyperbolic, the converse is not quite true: a geodesic metric space is $0$-hyperbolic if and only if it is an $\bR$-tree. We will
not discuss $\bR$-trees any further, since they are not necessary for our discussion.
We specialize the definition of a quasi-tree slightly: if $\gam$ is a graph
equipped with the graph metric,
we call it a quasi-tree if it contains a simplicial tree in its quasi-isometry class.

\begin{thm}[See~\cite{KK2013}]\label{thm:quasi-tree}
Let $\gam$ be a connected graph. Then $\gam^e$ is a quasi-tree, and is in particular $\delta$-hyperbolic. More precisely:
\begin{enumerate}
\item
If $\gam$ splits as a nontrivial join, then $\gam^e$ has finite diameter and is hence quasi-isometric to a point.
\item
If $\gam$ does not split as a nontrivial join then $\gam^e$ is quasi-isometric to a regular simplicial tree of countable degree.
\end{enumerate}
\end{thm}

More interesting than the mere description of the quasi-isometry type of the extension graph is the interaction between group elements
in $A(\gam)$ and $\gam^e$. Here, the analogy between the mapping class group and $A(\gam)$ develops further, with the natural
isometric action of $A(\gam)$ on $\gam^e$ mirroring many of the properties of the natural isometric action of $\Mod(\Sigma)$ on
$\CC(\Sigma)$.

The classical
\emph{Nielsen--Thurston classification}\index{Nielsen-Thurston classification}
~\cite{Thurston-bull88,FM2012} says that a mapping class is either finite order, reducible (i.e.~some power fixes the homotopy
class of an essential nonperipheral loop on the surface $\Sigma$), or pseudo-Anosov. As discussed around
Proposition~\ref{prop:gex-basic}, this lattermost type of mapping class is characterized by the fact that every orbit of its action
 on $\CC(\Sigma)$ is unbounded. Finite order and reducible mapping classes are characterized by every orbit in
$\CC(\Sigma)$ being bounded (and in fact having a periodic point in $\CC(\Sigma)$). Algebraically, a reducible mapping class has
a copy of $\bZ^2$ in its centralizer~\cite{BLM1983}, whereas a pseudo-Anosov mapping classes have virtually
cyclic centralizers~\cite{FLP1991}.

Further insight into the action of $\Mod(\Sigma)$ is provided by a result of Bowditch
\cite{Bowditch2008}, which says that the action of $\Mod(\Sigma)$ 
on $\CC(\Sigma)$ is \emph{acylindrical}\index{acylindrical action}.
Acylindricity is a notion of proper discontinuity for group actions on non-proper metric spaces
which are not properly discontinuous.
Following Bowditch (cf.~\cite{Sela-inv97,Koberda-whatis})
we say that an action of a group $G$ on a metric space $X$ is acylindrical if for all $r>0$ there
exist constants $R$ and $N$ such that for all pairs of points $x,y\in X$ with $d(x,y)\geq R$, we have
\[|\{g\in G\mid d(gx,x),d(gy,y)\leq r\}|\leq N.\] In other words, the $r$-quasi-stabilizer of $R$-separated points is uniformly finite.
Bowditch showed that if $X$ is a $\delta$-hyperbolic graph and $G$ acts
isometrically and acylindrically on $X$ then each $g\in G$ is either
\emph{elliptic}\index{elliptic isometry} or \emph{loxodromic}\index{loxodromic isometry}.
The former of these means that some (equivalently every) orbit of $G$ on $X$ is bounded.
A loxodromic element is characterized by having a positive asymptotic translation distance in $X$. Moreover, the
asymptotic translation length
is bounded away from zero by a constant that depends only on the hyperbolicity and acylindricity constants.
The Nielsen--Thurston classification can be thus recast in terms of acylindricity: a mapping class is pseudo-Anosov if and only if
it is loxodromic as an isometry of $\CC(\Sigma)$.

For extension graphs, one has a picture that is analogous to curve graphs.

\begin{thm}[See~\cite{KK2013b}]\label{thm:acyl}
Let $\gam$ be a connected graph with at least two vertices.
The action of $A(\gam)$ on $\gam^e$ is acylindrical. An element $1\neq g\in A(\gam)$ is elliptic if and only if $g$ is conjugate into
a subgroup $A(J)$, where $J$ is a subgraph of $\gam$ that is a nontrivial join. Equivalently, $g$ is elliptic if and only if its centralizer
in $A(\gam)$ is noncyclic.

An element $1\neq g\in A(\gam)$ is loxodromic if and only if its centralizer is cyclic. An element $g$ is cyclically reduced and loxodromic
if and only if $\supp(g)$ is not contained in a subgraph of $\gam$ that splits as a nontrivial join.
\end{thm}

The join/non-join dichotomy for graphs and their associated right-angled Artin groups runs deep,
and analogies between $A(\gam)$ and $\gam^e$ with $\Mod(\Sigma)$ and $\CC(\Sigma)$ are extensive. Many
(but not all; see~\cite{KMT2014}) of the instances of these analogies can be and have
 been incorporated into the theory of hierarchically hyperbolic groups.

A further equivalence in Theorem~\ref{thm:acyl} is given by a result of Behrstock--Charney~\cite{BC2010}, which asserts that a
nontrivial element of $A(\gam)$ is loxodromic if and only if, when viewed as a deck transformation of the universal cover of
the Salvetti complex $\mS(\gam)$, it acts as a \emph{rank one isometry}\index{rank one isometry}.
That is, the corresponding deck group element
has an axis that does not bound a half-plane (cf.~\cite{BH1999}).

\subsection{The extension graph as a quasi-isometry and commensurability invariant}

A basic problem in geometric group theory is to sort groups into quasi-isometry classes. For right-angled Artin groups, the natural
question is to decide when two right-angled Artin groups $A(\gam)$ and $A(\Lambda)$ are quasi-isometric. Much progress on understanding
the quasi-isometric classification of right-angled Artin groups has been made,
for instance by Behrstock--Neumann~\cite{BehrNeu08}, Behrstock--Januszkiewicz--Neumann~\cite{BJN2010},
Bestvina--Kleiner--Sageev~\cite{BKS2008},
Huang~\cite{Huang18}, and Margolis~\cite{Margolis20} (see also~\cite{HuangKlein18,Casals16-agt}).
Thus, we can consider the following equivalence relation on finite graphs:
$\gam$ is equivalent to $\Lambda$ if $A(\gam)$ and $A(\Lambda)$
are quasi-isometric to each other. Other than the cases we have cited, understanding this equivalence relation in full is still unresolved.

Certainly two right-angled Artin groups that are isomorphic to each other will be quasi-isometric to each other, and from
Theorem~\ref{thm:coho-alg}, we know that if $A(\gam)$ and $A(\Lambda)$ are isomorphic to each other then $\gam$ and $\Lambda$
are isomorphic as graphs. There is yet another equivalence relation on finite graphs that is coarser than isomorphism and yet finer
than quasi-isometry.

If $H<G$ are groups with $G$ finitely generated and $[G:H]<\infty$, then with respect to any finite generating sets for $G$ and $H$, the
inclusion of $H$ into $G$ is a quasi-isometry on the level of Cayley graphs, as is readily verified. It follows that if $G$ and $H$ are
finitely generated groups,
and both $G$ and $H$ contain a finite index subgroup isomorphic to $K$, then $G$ and $H$ are quasi-isometric. In this case, we say
that $G$ and $H$ are \emph{commensurable}\index{commensurable}.
Like quasi-isometry, commensurability is an equivalence relation on groups.
It is well known that commensurability of groups is a strictly finer equivalence relation than
quasi-isometry. For instance, one can take closed hyperbolic $3$-manifolds whose volumes are not rational multiples of each other. Then,
the corresponding fundamental groups are both quasi-isometric to hyperbolic space, but are not commensurable
~\cite{GdlH88}. Even among
right-angled Artin groups, commensurability is a strictly finer equivalence relation (see~\cite{Huang18,CKZ19}).

It is easy to produce pairs of non-isomorphic graphs which give rise to commensurable right-angled Artin groups. For instance, consider
a graph $\gam$ and $v\in V(\gam)$. There is a 
surjective homomorphism $A(\gam)\longrightarrow\bZ/2\bZ$ that sends $v$ to the nontrivial element in
$\bZ/2\bZ$ and sends the remaining vertices to the identity. It is an exercise in combinatorial group theory for the reader to prove that
the kernel of this homomorphism is isomorphic to $A(\Lambda)$, where $\Lambda$ is obtained by taking two copies of $\gam$ and
identifying them along $\st(v)$. If $v$ is not central in $A(\gam)$ then it is easy to see that $\Lambda$ and $\gam$ fail to be isomorphic
graphs,
but $A(\gam)$ and $A(\Lambda)$ are clearly commensurable. This construction can be repeated \emph{ad infinitum}, generally producing
infinite families of non-isomorphic graphs whose associated right-angled Artin groups are all commensurable.

There are pairs of graphs which give rise to commensurable right-angled Artin groups, but for which a commensuration between them is less
obvious. The reader is challenged to prove for themself that the groups $A(P_4)$ and $A(P_5)$ are commensurable, where
as before $P_4$ and $P_5$ denote the paths of length three and length four respectively
(cf.~\cite{CKZ19}). The fact that $A(P_4)$ and $A(P_5)$ are
commensurable also shows that the extension graph is hopeless as a complete commensurability invariant. Again, the reader is
encouraged to convince themself that the extension graphs of $P_4$ and $P_5$ are not isomorphic to each other. It turns out that
in both cases, the corresponding extension graphs are trees, and what distinguishes them in their isomorphism type is the location
of degree one vertices.

So, let us consider a
connected graph $\gam$ with no degree one vertices. In order to identify the extension graph algebraically and in an
unambiguous way, it would help to be able to identify vertices and their conjugates, up to powers. For this, it helps to assume that
$\gam$ is connected, has no triangles, and has no squares. Under these assumptions, if $v$ is a vertex of $\gam$ then $v$ contains
a nonabelian free group in its centralizer. Conversely, suppose that $g\in A(\gam)$ has a nonabelian free group in its centralizer.
Then, since $\gam$ has no triangles and no squares, every nontrivial join in $\gam$ is merely the star of a vertex of $\gam$, and
the structure of such a star is the join of a single vertex and a completely disconnected graph. It follows that if $g$
has a nonabelian free group in its centralizer, then $g$ is conjugate to a nonzero power of a vertex generator of $\gam$. It follows
that maximal cyclic subgroups of $A(\gam)$ whose centralizers contain nonabelian free groups are in bijection with conjugates of 
vertex generators of $A(\gam)$. Since the adjacency relation in $\gam^e$ is just commutation in $A(\gam)$, we immediately obtain:

\begin{thm}[See~\cite{KK2013b}]
Let $\gam$ be a finite
connected graph with no degree one vertices, no triangles, and no squares. Then the extension graph $\gam^e$ is
a commensurability invariant for $A(\gam)$. That is, if $A(\gam)$ is commensurable with $A(\Lambda)$ then $\gam^e\cong
\Lambda^e$.
\end{thm}

Incidentally, the analogy between right-angled Artin groups and mapping class groups persists here as well, since the
curve graph can be obtained from the mapping class group in the same way that the extension graph is obtained from $A(\gam)$.
Specifically, let $T$ be a Dehn twist about a simple closed curve on $\Sigma$. Then $T$ is centralized by two maximal
rank torsion-free abelian subgroups
of $\Mod(\Sigma)$ which intersect in a copy of $\bZ$. This can be used to algebraically characterize a
(nonzero power of a) Dehn twist as an element of
$\Mod(\Sigma)$. A Dehn twist unambiguously identifies the homotopy class of a simple closed curve on $\Sigma$, and
the adjacency relation
in $\CC(\Sigma)$ coincides with commutation of Dehn twists in $\Mod(\Sigma)$. Thus, the curve graph can be recovered
algebraically from $\Mod(\Sigma)$. It follows in particular that
automorphisms of $\Mod(\Sigma)$ induce automorphisms of $\CC(\Sigma)$, a fact which can be used to prove various rigidity results
(see~\cite{ivanov,BrenMar04,LeinMar06}, for instance).

As we have seen, we can have commensurable right-angled Artin groups with non-isomorphic extension graphs. It is also possible
to have two right-angled Artin groups whose extension graphs are isomorphic
and yet the groups are not quasi-isometric to each other (see Example 5.22 in~\cite{Huang2016}).
So, there is similarly no hope that extension graphs
form a complete quasi-isometry invariant for right-angled Artin groups.

Recall from Subsection~\ref{ss:aut} that a full set of generators for $\Aut(A(\gam))$ is known, and from the description of
these generators, it is immediate that $\Out(A(\gam))$ is finite if and only if $\Aut(A(\gam))$ admits no nontrivial partial conjugations
and no dominated transvections. Graphs for which $\Out(A(\gam))$ is finite can thus be identified through a finitary combinatorial
analysis, since it suffices to check that there are no separating stars of vertices and no pairs of vertices where one dominates the
other (see~\cite{charney-farber} for a discussion of the genericity of this phenomenon).

The following result was established by Huang~\cite{Huang17}:

\begin{thm}
Suppose $\gam$ is a graph for which $\Out(A(\gam))$ is finite. The following are equivalent:
\begin{enumerate}
\item
The group $A(\Lambda)$ is quasi-isometric to $A(\gam)$.
\item
The group $A(\Lambda)$ is isomorphic to a finite index subgroup of $A(\gam)$.
\item
The graphs $\Lambda^e$ and $\gam^e$ are isomorphic.
\end{enumerate}
\end{thm}

Thus, in the case of finite groups of outer automorphisms, quasi-isometry, commensurability, and isomorphism of extension graphs
are equivalent conditions to place on a right-angled Artin group. Here again, the analogy with mapping class groups persists.
If two mapping class groups of surfaces are quasi-isometric, then except for some sporadic cases, the resulting mapping class
groups are in fact isomorphic to each other~\cite{BKMM2012}. Thus again excluding some sporadic cases,
quasi-isometry, commensurability, and isomorphism of
mapping class groups are equivalent. Finally, aside from some sporadic cases, isomorphism of curve graphs is equivalent to
isomorphism of mapping class groups~\cite{Shackleton07}.

\section{Further directions}

Much remains to be understood in the relationship between combinatorics and algebra via the lens of right-angled Artin groups.
As the reader has certainly come to understand, it is not just some property of groups that one seeks to analogize a property of graphs;
one wants it to be a clean and natural statement about groups that reflects the particular flavor of the property in question.
Therefore, it is not likely one could produce a satisfactory omnibus result,
since some subjective notions of beauty and philosophical considerations enter into the picture.

With these musings, we close by giving some particular open questions of interest. Some are well-known open problems, and we make
no claim to having been the first to pose them.

\begin{que}
What is the full quasi-isometric classification of right-angled Artin groups? What about the commensurability classification of
right-angled Artin groups? What sorts of combinatorial objects serve as complete invariants for these equivalence relations?
\end{que}

Some specific natural combinatorial properties we have not discussed are of interest in graph theory.

\begin{que}\label{que:planar}
What algebraic property of $A(\gam)$ is equivalent to the planarity of $\gam$?
\end{que}

Closely related to Question~\ref{que:planar} is the problem of determining whether a graph $\Lambda$ is a subdivision of a graph
$\gam$ by examining the relationship between the groups $A(\Lambda)$ and $A(\gam)$, which to the knowledge of the author is also open.

A graph is \emph{self-complementary}\index{self-complementary graph}
if it is isomorphic to its complement graph. A singleton vertex is self-complementary, as are the path $P_4$ of length
three and the cycle $C_5$ of length five. A question that is a particular favorite of the author is the following:

\begin{que}
What algebraic property of $A(\gam)$ is equivalent to the statement that $\gam$ is self-complementary?
\end{que}

Following the remarks in Subsection~\ref{ss:further} above and the results of~\cite{HermillerSunic}, we have the following.

\begin{que}
What is the relationship between the normal subgroup structure of $A(\gam)$ and the combinatorics of $\gam$?
\end{que}

Finally, we have the following more open-ended question.

\begin{que}\label{que:spectral}
Is there a synthesis between the ideas in Section~\ref{sec:comb} and algebraic graph theory? How can one formulate spectral graph
theory in terms of right-angled Artin groups?
\end{que}

Some of the discussion in this survey is a step towards an answer to Question~\ref{que:spectral}. For one, the Cheeger constant $c$
of a graph
can be viewed as a spectral invariant of a graph, as it controls the spectral gap of the the graph via the 
Cheeger inequality due to Dodziuk and Alon--Milman (see~\cite{KowalskiBook} for a detailed discussion):
if $\lambda_2$ is the second largest eigenvalue of a $d$--regular connected graph $\gam$
then \[\frac{1}{2}(d-\lambda_2)\leq c\leq \sqrt{2d(d-\lambda_2)}.\] The content
of Theorem~\ref{thm:expander} is that the Cheeger constant of a graph $\gam$ can be read off from the cohomology algebra of $A(\gam)$.
It is natural to ask how one might recover more information about the eigenvalues of the adjacency matrix of $\gam$ from the
group theory of $A(\gam)$.

We hope that this survey will encourage further investigations in these directions.

\section*{Acknowledgements}

The author is partially supported  by an Alfred P. Sloan Foundation Research Fellowship, by NSF Grant DMS-1711488,
and by NSF Grant DMS-2002596. The author thanks K.~Ohshika and A.~Papadopoulos for inviting him to write this survey, and to
R.~Flores for providing numerous invaluable comments on an earlier draft.


\bibliographystyle{amsplain}

\end{document}